\documentclass[11pt]{amsart}

\usepackage{amsmath,amssymb,latexsym,soul,cite,mathrsfs}

\usepackage{color,enumitem,graphicx}
\usepackage[colorlinks=true,urlcolor=blue,
citecolor=red,linkcolor=red,linktocpage,pdfpagelabels,
bookmarksnumbered,bookmarksopen]{hyperref}
\usepackage[english]{babel}
\usepackage{esint}

\usepackage[left=2.9cm,right=2.9cm,top=2.8cm,bottom=2.8cm]{geometry}
\usepackage[hyperpageref]{backref}

\usepackage[colorinlistoftodos]{todonotes}
\makeatletter
\providecommand\@dotsep{5}
\def\listtodoname{List of Todos}
\def\listoftodos{\@starttoc{tdo}\listtodoname}
\makeatother

\numberwithin{equation}{section}

\newcommand{\Lp}[1]{L^{#1}(\Omega)}
\newcommand{\Wpzero}[1]{W^{1,#1}_0(\Omega)}
\newcommand*\diff{\mathrm{d}}
\newcommand\V{W^{1,\mathcal{H}}_0(\Omega)}

\newcommand{\R}{\mathbb{R}}

\newcommand{\Om} {\Omega}
\newcommand{\la} {\lambda}

\newcommand{\mb} {\mathbb}
\newcommand{\mc} {\mathcal}

\newtheorem{Theorem}{Theorem}[section]
\newtheorem{Lemma}[Theorem]{Lemma}
\newtheorem{Proposition}[Theorem]{Proposition}
\newtheorem{Corollary}[Theorem]{Corollary}
\newtheorem{Remark}[Theorem]{Remark}
\newtheorem{Definition}[Theorem]{Definition}

\begin{document}

\title[Anisotropic double phase singular problem with sign changing nonlinearity]
{On an anisotropic double phase problem with singular and sign changing nonlinearity}
\author{Prashanta Garain and Tuhina Mukherjee}

\address[Prashanta Garain ]
{\newline\indent Department of Mathematics
	\newline\indent
Uppsala University
\newline\indent
Uppsala-75106, Sweden
\newline\indent
Email: {\tt pgarain92@gmail.com} }

\address[Tuhina Mukherjee ]
{\newline\indent Department of Mathematics
\newline\indent
Indian Institute of Technology Jodhpur
\newline\indent
Rajasthan-342030, India
\newline\indent
Email: {\tt tuhina@iitj.ac.in} }

\begin{abstract}
This article consists of study of anisotropic double phase problems with singular term and sign changing subcritical as well as critical nonlinearity. Seeking the help of well known Nehari manifold technique, we establish existence of at least two opposite sign energy solutions in the subcritical case and one negative energy solution in the critical case. The results in the critical case is even new in the classical $p$-Laplacian case. 
\end{abstract}



\maketitle

\section{Introduction}
In this article, we study existence of weak solutions for the following singular anisotropic double phase problem
\begin{equation}\label{maineqn}
\begin{split}
-H_{p,q}u&=a(x)u^{-\delta}+\lambda b(x)u^{\alpha}\text{ in }\Omega,\\
 u&>0\text{ in }\Omega,\\
 u&=0\text{ on }\partial\Om,
\end{split}
\end{equation}
where $\Om$ is a bounded Lipschitz domain in $\mathbb{R}^N$, $N\geq 2$. The operator
\begin{equation}\label{fpq}
H_{p,q}u=\text{div}(H(\nabla u)^{p-1}\nabla_{\eta}H(\nabla u)+\mu(x)H(\nabla u)^{q-1}\nabla_{\eta}H(\nabla u))
\end{equation}
is referred to as the double phase anisotropic $(p,q)$-Laplace operator, where $\nabla_{\eta}H(\nabla u)$ denotes the gradient of the Finsler-Minkowski norm $H$ (defined in Section 2) with respect to $\nabla u$.

Unless otherwise stated, throughout the paper, we assume that 
\begin{enumerate}
\item[(H1)] $\la>0$,\,\,$0<\delta<1$, $2\leq p<q<N$,\,\,$q<\alpha+1\leq p^*=\frac{Np}{N-p}$, $\frac{q}{p}< 1+ \frac{1}{N}$ and $0\leq \mu(\cdot)\in C^{0,1}(\overline{\Om})$.
\end{enumerate}
In the subcritical case, that is when $\alpha+1<p^{*}$, we assume that
\begin{enumerate}
\item[(H2)] $a>0$ a.e. in $\Om$ such that $a\in L^\infty(\Om)$ and $b\in L^{\frac{p^*}{p^*-\alpha-1}}(\Omega)$ may change sign in $\Om$.
\end{enumerate}
Moreover, in the critical case $\alpha+1=p^{*}$, we assume that
\begin{enumerate}
    \item[(H3)] $a>0$ a.e. in $\Om$ and $a\in L^{\frac{p^*}{p^*-1+\delta}}(\Om)$,
    \item[(H4)] $b$ is possibly sign-changing a.e. in $\Om$ with $b^+\not\equiv 0$ and $b\in L^\infty(\Om)$.
\end{enumerate}
Equation \eqref{maineqn} is closely associated with the double phase functional
$$
v\longmapsto \int_{\Om}(H(\nabla v)^p+\mu(x)H(\nabla v)^q)\,dx,
$$
where the integrand $F(x,\zeta)=H(\zeta)^p+\mu(x)(x)H(\zeta)^q$ satisfies the non-standard growth condition
$$
k_1|\zeta|^p\leq k_2 H(\zeta)^p\leq F(x,\zeta)\leq k_3(1+H(\zeta)^q)\leq k_4(1+|\zeta|^q),
$$
for a.e. $x\in\Om$ and for all $\zeta\in\mathbb{R}^N$, where $k_i>0$ for all $i=1,\cdots,4$. In the classical case $H(\zeta)=|\zeta|$, such  functional was first introduced by Zhikov \cite{Zhi1, Zhi2} which has a wide range of applications in material science, see Zhikov-Kozlov-Oleinik \cite{Zhi3} and the references therein. The double phase behavior of such functional is captured by the weight function $\mu(\cdot)$. More precisely, the $p$-growth term governs the functional over the set $\{\mu(x)=0\}$ and the $q$-growth term governs the functional over $\{\mu(x)\neq 0\}$ which motivates \eqref{maineqn} to be referred as a double phase equation. Further, a colossal amount of literature can be found, for example, in Colombo-Mingeone \cite{Min2, Min3}, Baroni-Kuusi-Mingione \cite{BKM1}, Marcellini \cite{Mar1, Mar2}, Ok \cite{Ok1, Ok2}, Ragusa-Tachikawa \cite{Ragu1} and the references therein.

We call \eqref{maineqn} as a singular problem due to appearance of the term $u^{-\delta}$ for $\delta\in(0,1)$, which clearly blows up near the boundary(origin). Such type of phenomenon in the context of double phase equation has captured the attention of researchers recently, refer Liu-Dai-Papageorgiou-Winkert \cite{win}. More precisely, when $F(\zeta)=|\zeta|$, existence of multiple weak solutions has been proved in \cite{win} for the singular double phase equation in the subcrtitical case but the nonlinearities were not sign-changing.

Although it is worth mentioning that anisotropic singular problem is very less understood. When $a(x)=0$, singular anisotropic problems is studied in Biset-Mebrate-Mohammed \cite{BMM20}, Farkas-Winkert \cite{PF20} and Farkas-Fiscella-Winkert \cite{PF21}, Bal-Garain-Mukherjee \cite{BGM}. To the best of our knowledge, in the double phase context, anisotropic singular problems has been first discussed in Farkas-Winkert \cite{Win2}, where the authors proved existence of one weak solution in the critical case.


In this article, our main motive is to study the equation \eqref{maineqn} in both the subcritical and critical case. Another striking feature in the equation \eqref{maineqn} lies in the sign changing behavior of the nonlinearity. To be more precise, in the subcritical case $\alpha+1<p^*$, we prove existence of multiple weak solutions (see Theorem \ref{scthm}), whereas in the critical case $\alpha+1=p^*$, we prove existence of one weak solution (see Theorem \ref{cthm}) of the problem \eqref{maineqn}. To prove our main results, we follow the Nehari manifold approach where we minimized the energy functional over suitable subsets of $W^{1,\mathcal H}_0(\Om)$ and considerably established them as weak solutions of our problem \eqref{maineqn}.  Our aim behind this article is to develop the Nehari manifold theory for double phase anisotropic singular problems having sign changing weight function which is completely new as per current literature available. We also remark that the study of \eqref{maineqn} in critical case for  the classical $p$-Laplacian case for sign-changing weight functions is not studied till now. So we provide an interesting result in this regard.

Due to the double phase structure of our equation \eqref{maineqn}, we investigate the solutions in the Musielak-Orlicz Sobolev space $W_0^{1,\mathcal{H}}(\Om)$ (see Section 2 for more details). The notion of weak solution for the equation \eqref{maineqn} is defined as follows:
\begin{Definition}\label{wksol}
We say that $u\in W_0^{1,\mathcal{H}}(\Om)$ is a weak solution of the problem \eqref{maineqn}, if $u>0$ a.e. in $\Om$ and for every $\phi\in C_c^{1}(\Om)$, we have
\begin{equation}\label{wksoleqn}
\begin{split}
&\int_{\Om}(H(\nabla u)^{p-1}+\mu(x)H(\nabla u)^{q-1})\nabla_{\eta}H(\nabla u)\nabla\phi\,dx=\int_{\Om}(a(x)u^{-\delta}+\la b(x)u^{\alpha})\phi\,dx.
\end{split}
\end{equation}
\end{Definition}
For $\lambda>0$, we define the energy functional $I_{\lambda}: W_{0}^{1,\mathcal{H}}(\Om)\to\R$ associated with the problem \eqref{maineqn} by 
\begin{equation}\label{engfnl*}
\begin{split}
I_{\lambda}(u)&=\int_{\Om}\Big(\frac{H(\nabla u)^p}{p}+\mu(x)\frac{H(\nabla u)^q}{q}\Big)\,dx-\frac{1}{1-\delta}\int_{\Om}a(x)|u|^{1-\delta}\,dx-\frac{\lambda}{\alpha+1}\int_{\Om}b(x)|u|^{\alpha+1}\,dx.
\end{split}
\end{equation}
We say that $u\in W_0^{1,\mathcal{H}}(\Omega)$ has a positive (resp. negative) energy if $I_{\la}(u)>0$ (resp. $<0$). 
\subsection{Statement of the Main results:} Our main results in this article reads as follows:
\begin{Theorem}\label{scthm}
Suppose that $(H1)$ and $(H2)$ are satisfied for $\alpha+1<p^*$. Then there exists a $\Lambda>0$ such that for every $\la\in(0,\Lambda)$, the problem \eqref{maineqn} admits at least two distinct weak solutions in $W_0^{1,\mathcal{H}}(\Om)$, one with positive and other with negative energy. 
\end{Theorem}

\begin{Theorem}\label{cthm}
Suppose that $(H1),\,(H3)$ and $(H4)$ are satisfied for $\alpha=p^*-1$. Then there exists a $\Lambda_{**}>0$ such that for every $\la\in(0,\Lambda_{**})$, the problem \eqref{maineqn} admits at least one negative energy weak solution.
\end{Theorem}

\subsection{Functional setting and auxiliary results}
Here, $H$ is the Finsler-Minkowski norm, that is $H:\mathbb{R}^N\to[0,\infty)$ satisfies the following hypothesis:
\begin{enumerate}
\item[(A1)] $H(x)\geq 0$ for all $x\in\mathbb{R}^N$.
\item[(A2)] $H(x)=0$ if and only if $x=0$.
\item[(A3)] $H\in C^{\infty}(\mathbb{R}^N\setminus\{0\})$.
\item[(A4)] $H(tx)=|t|H(x)$ for all $x\in\mathbb{R}^N$ and $t\in\mathbb{R}$. 
\item[(A5)] The Hessian matrix $\nabla ^2(\frac{H^2}{2})(x)$ is positive definite at every $x\neq 0$.
\end{enumerate}

In this section, we will recall the main properties and embedding results for Musielak-Orlicz Sobolev spaces. To this end, we suppose that $\Omega\subset \R^N$ ($N\geq 2$) is a bounded domain with Lipschitz boundary $\partial\Omega$. For any $r\in [1,\infty)$, we denote by $\Lp{r}=L^r(\Omega;\R)$ and $L^r(\Omega;\R^N)$ the usual Lebesgue spaces with the norm $\|\cdot\|_r$. Moreover, the Sobolev space $\Wpzero{r}$ is equipped with the equivalent norm $\|\nabla \cdot \|_r$ for $1<r<\infty$. Let hypothesis \textnormal{(H1)} be satisfied and consider the nonlinear function $\mathcal{H}\colon \Omega \times [0,\infty)\to [0,\infty)$ defined by
\begin{align*}
	\mathcal H(x,t)= t^p+\mu(x)t^q.
\end{align*}
Denoting by $M(\Omega)$ the space of all measurable functions $u\colon\Omega\to\R$, we can introduce the Musielak-Orlicz Lebesgue space $L^\mathcal{H}(\Omega)$ which is given by
\begin{align*}
	L^\mathcal{H}(\Omega)=\left \{u\in M(\Omega)\,:\,\varrho_{\mathcal{H}}(u)<\infty \right \}
\end{align*}
equipped with the Luxemburg norm
\begin{align*}
	\|u\|_{\mathcal{H}} = \inf \left \{ \tau >0\,:\, \varrho_{\mathcal{H}}\left(\frac{u}{\tau}\right) \leq 1  \right \},
\end{align*}
where the modular function $\varrho_{\mathcal{H}}:L^\mathcal{H}(\Omega)\to\mathbb{R}$ is given by
\begin{align*}
	\varrho_{\mathcal{H}}(u):=\int_\Omega \mathcal{H}(x,|u|)\,\diff x=\int_\Omega \big(|u|^{p}+\mu(x)|u|^q\big)\,\diff x.
\end{align*}

The norm $\|\,\cdot\,\|_{\mathcal{H}}$ and the modular function $\varrho_{\mathcal{H}}$ have the following relations, see Liu-Dai \cite[Proposition 2.1]{Liu-Dai-2018} or Crespo-Blanco-Gasi\'nski-Harjulehto-Winkert \cite[Proposition 2.13]{Crespo-Blanco-Gasinski-Harjulehto-Winkert-2021}.

\begin{Proposition}\label{proposition_modular_properties}
	Let \textnormal{(H1)} be satisfied, $u\in L^{\mathcal{H}}(\Omega)$ and $c>0$. Then the following hold:
	\begin{enumerate}
		\item[\textnormal{(i)}]
			If $u\neq 0$, then $\|u\|_{\mathcal{H}}=c$ if and only if $ \varrho_{\mathcal{H}}(\frac{u}{c})=1$.
		\item[\textnormal{(ii)}]
			$\|u\|_{\mathcal{H}}<1$ (resp.\,$>1$, $=1$) if and only if $ \varrho_{\mathcal{H}}(u)<1$ (resp.\,$>1$, $=1$).
		\item[\textnormal{(iii)}]
			If $\|u\|_{\mathcal{H}}<1$, then $\|u\|_{\mathcal{H}}^q\leq \varrho_{\mathcal{H}}(u)\leq\|u\|_{\mathcal{H}}^p$.
		\item[\textnormal{(iv)}]
			If $\|u\|_{\mathcal{H}}>1$, then $\|u\|_{\mathcal{H}}^p\leq \varrho_{\mathcal{H}}(u)\leq\|u\|_{\mathcal{H}}^q$.
		\item[\textnormal{(v)}]
			$\|u\|_{\mathcal{H}}\to 0$ if and only if $ \varrho_{\mathcal{H}}(u)\to 0$.
		\item[\textnormal{(vi)}]
			$\|u\|_{\mathcal{H}}\to \infty$ if and only if $ \varrho_{\mathcal{H}}(u)\to \infty$.
	\end{enumerate}
\end{Proposition}

Furthermore, we define the seminormed space
\begin{align*}
	L^q_\mu(x)(\Omega)=\left \{u\in M(\Omega)\,:\,\int_\Omega \mu(x) |u|^q \,\diff x< \infty \right \}
\end{align*}
endowed with the seminorm
\begin{align*}
	\|u\|_{q,\mu} = \left(\int_\Omega \mu(x) |u|^q \,\diff x \right)^{\frac{1}{q}}.
\end{align*}
While, the corresponding Musielak-Orlicz Sobolev space $W^{1,\mathcal{H}}(\Omega)$ is defined by
\begin{align*}
	W^{1,\mathcal{H}}(\Omega)= \Big \{u \in L^\mathcal{H}(\Omega) \,:\, H(\nabla u) \in L^{\mathcal{H}}(\Omega) \Big\}
\end{align*}
equipped with the norm
\begin{align*}
	\|u\|_{1,\mathcal{H}}= \|H(\nabla u )\|_{\mathcal{H}}+\|u\|_{\mathcal{H}}.
\end{align*}
 Moreover, we denote by $W^{1,\mathcal{H}}_0(\Omega)$ the completion of $C^\infty_0(\Omega)$ in $W^{1,\mathcal{H}}(\Omega)$. From hypothesis \textnormal{(H1)}, we know that we can equip the space $\V$ with the equivalent norm given by
\begin{align*}
	\|u\|=\|H(\nabla u)\|_{\mathcal{H}},
\end{align*}
see Proposition  2.18 of Colasuonno-Squassina \cite{Colasuonno-Squassina-2016}. It is known that $L^\mathcal{H}(\Omega)$, $W^{1,\mathcal{H}}(\Omega)$ and $W^{1,\mathcal{H}}_0(\Omega)$ are uniformly convex and so reflexive Banach spaces, see Colasuonno-Squassina \cite[Proposition 2.14]{Colasuonno-Squassina-2016}.
We end this section by recalling the following embeddings for the spaces $L^\mathcal{H}(\Omega)$ and $W^{1,\mathcal{H}}_0(\Omega)$, see Colasuonno-Squassina \cite[Proposition 2.15]{Colasuonno-Squassina-2016}.

\begin{Proposition}\label{proposition_embeddings}
	Let \textnormal{(H1)} be satisfied. Then the following embeddings hold:
	\begin{enumerate}
		\item[\textnormal{(i)}]
			$\Lp{\mathcal{H}} \hookrightarrow \Lp{r}$ and $\V\hookrightarrow \Wpzero{r}$ are continuous for all $r\in [1,p]$.
		\item[\textnormal{(ii)}]
			$\V \hookrightarrow \Lp{r}$ is continuous for all $r \in [1,p^*]$ and compact for all $r \in [1,p^*)$.
		\item[\textnormal{(iii)}]
			$\Lp{\mathcal{H}} \hookrightarrow L^q_\mu(\Omega)$ is continuous.
		\item[\textnormal{(iv)}]
			$\Lp{q}\hookrightarrow\Lp{\mathcal{H}} $ is continuous.
	\end{enumerate}
\end{Proposition}

\begin{Remark}
	Since $q<p^*$ by hypothesis \textnormal{(H1)}, we know from Proposition \ref{proposition_embeddings}\,\textnormal{(ii)} that $\V \hookrightarrow L^q(\Omega)$ is compact.
\end{Remark}

\section{Preliminaries}
\subsection{Nehari set}
We prove existence of weak solutions of \eqref{maineqn} by variational methods. For $\lambda>0$, we define the energy functional $I_{\lambda}: W_{0}^{1,\mathcal{H}}(\Om)\to\R$ associated with the problem \eqref{maineqn} by 
\begin{equation}\label{engfnl}
\begin{split}
I_{\lambda}(u)&=\int_{\Om}\Big(\frac{H(\nabla u)^p}{p}+\mu(x)\frac{H(\nabla u)^q}{q}\Big)\,dx-\frac{1}{1-\delta}\int_{\Om}a(x)|u|^{1-\delta}\,dx\\
&\quad-\frac{\lambda}{\alpha+1}\int_{\Om}b(x)|u|^{\alpha+1}\,dx.
\end{split}
\end{equation}
Our approach is based on the Nehari manifold technique. To this end, we define the following subsets of $W_0^{1,\mathcal{H}}(\Om)$ as follows:
\begin{equation}\label{Nset}
M_{\lambda}=\{u\in W_0^{1,\mathcal{H}}(\Om)\setminus\{0\}:\langle I_\la^\prime(u),u\rangle=0\}.
\end{equation}
For $t>0$ and $u\in W_0^{1,\mathcal{H}}(\Om)$, we define the fibering map $\phi_u:(0,\infty)\to\R$ by $\phi_u(t)=I_{\lambda}(tu)$. It is easy to deduce the following:
\begin{align*}
    \phi_u(t) &= \int_{\Om}\Big(\frac{t^pH(\nabla u)^p}{p}+\mu(x)\frac{t^qH(\nabla u)^q}{q}\Big)\,dx-\frac{t^{1-\delta}}{1-\delta}\int_{\Om}a(x)|u|^{1-\delta}\,dx\\
    &\quad -\frac{\lambda t^{\alpha+1}}{\alpha+1}\int_{\Om}b(x)|u|^{\alpha+1}\,dx.\\
    \phi_u^\prime(t) &= \int_{\Om}\Big({t^{p-1}H(\nabla u)^p}+\mu(x){t^{q-1}H(\nabla u)^q}\Big)\,dx-{t^{-\delta}}\int_{\Om}a(x)|u|^{1-\delta}\,dx\\
    &\quad -{\lambda t^{\alpha}}\int_{\Om}b(x)|u|^{\alpha+1}\,dx.\\
    \phi_u^{\prime\prime}(t) &= \int_{\Om}\Big({(p-1)t^{p-2}H(\nabla u)^p}+\mu(x){(q-1)t^{q-2}H(\nabla u)^q}\Big)\,dx+\delta{t^{-\delta-1}}\int_{\Om}a(x)|u|^{1-\delta}\,dx\\
    &\quad -{\lambda\alpha
    t^{\alpha-1}}\int_{\Om}b(x)|u|^{\alpha+1}\,dx.
\end{align*}

We split the set $M_{\la}$ as $M_\lambda=M_{\lambda}^{0}\cup M_{\lambda}^{+}\cup M_{\lambda}^{-}$, where
\begin{equation}\label{Nset0}
M_{\lambda}^{0}=\{u\in M_{\lambda}:\phi_u^{''}(1)=0\},
\end{equation}
\begin{equation}\label{Nset+}
M_{\lambda}^{+}=\{u\in M_{\lambda}:\phi_u^{''}(1)>0\}
\end{equation}
and
\begin{equation}\label{Nset-}
M_{\lambda}^{-}=\{u\in M_{\lambda}:\phi_u^{''}(1)<0\}.
\end{equation}

{Our first lemma asserts the coercivity and existence of infimum of $I_{\lambda}$ defined by \eqref{engfnl} over the Nehari set $M_\lambda$.
\begin{Lemma}\label{cbb}
The functional $I_\lambda$ is coercive and bounded below over $M_{\lambda}$.
\end{Lemma}
\begin{proof}
Since $u\in M_{\lambda}$, we have
\begin{align*}
I_{\lambda}(u)&=\Big(\frac{1}{p}-\frac{1}{\alpha+1}\Big)\int_{\Om}H(\nabla u)^p\,dx+\Big(\frac{1}{q}-\frac{1}{\alpha+1}\Big)\int_{\Om}H(\nabla u)^q\,dx\\
&\qquad-\Big(\frac{1}{1-\delta}-\frac{1}{\alpha+1}\Big)\int_{\Om}a(x)|u|^{1-\delta}\,dx.
\end{align*}
By Proposition \ref{proposition_embeddings}-(ii), we have $\int_{\Om}a(x)|u|^{1-\delta}\,dx\leq c\|u\|^{1-\delta}$ for some constant $c>0$. Therefore, since $p<q$, we have
\begin{align*}
I_{\lambda}(u)&\geq\Big(\frac{1}{q}-\frac{1}{\alpha+1}\Big)\int_{\Om}(H(\nabla u)^p\,dx+H(\nabla u)^q)\,dx-c\Big(\frac{1}{1-\delta}-\frac{1}{\alpha+1}\Big)\|u\|^{1-\delta}\\
&=\Big(\frac{1}{q}-\frac{1}{\alpha+1}\Big)\rho_{\mathcal{H}(\nabla u)}-c\Big(\frac{1}{1-\delta}-\frac{1}{\alpha+1}\Big)\|u\|^{1-\delta}.
\end{align*}
If $\|u\|>1$, by Proposition \ref{proposition_modular_properties}, we have $\rho_{\mathcal{H}}(\nabla u)\geq\|u\|^{p}$. Using this property in the above inequality, the result follows.
\end{proof}

Next, we prove that for small values of $\lambda$, the set $M_{\lambda}^{0}$ is empty. 

\begin{Lemma}\label{Nest0emp}
There exists $\lambda_0>0$ such that $M_\lambda^{0}=\emptyset$, for every $\lambda\in(0,\lambda_0)$.
\end{Lemma}
\begin{proof}
Suppose for every $\lambda^*>0$, there exists $\la\in(0,\la^*)$ such that $M_{\lambda}^{0}\neq\emptyset$. Hence for any given $\la>0$, let $u\in M_{\lambda}^{0}$. Then we have $\phi_u^{'}(1)=0$ and $\phi_u^{''}(1)=0$, which gives
\begin{equation}\label{p1}
\int_{\Om}(H(\nabla u)^p+\mu(x)H(\nabla u)^q)\,dx=\int_{\Om}a(x)|u|^{1-\delta}\,dx+\lambda\int_{\Om}b(x)|u|^{\alpha+1}\,dx
\end{equation}
and
\begin{equation}\label{p2}
\int_{\Om}((p-1)H(\nabla u)^p+(q-1)\mu(x)H(\nabla u)^q)\,dx+\delta\int_{\Om}a(x)|u|^{1-\delta}\,dx=\lambda\alpha\int_{\Om}b(x)|u|^{\alpha+1}\,dx
\end{equation}
respectively. Putting the value of $\int_{\Om}a(x)|u|^{1-\delta}\,dx$ from \eqref{p1} in \eqref{p2}, we obtain
\begin{equation}\label{p3}
(p+\delta-1)\int_{\Om}H(\nabla u)^p\,dx+(q+\delta-1)\int_{\Om}\mu(x)H(\nabla u)^q\,dx=\lambda(\alpha+\delta)\int_{\Om}b(x)|u|^{\alpha+1}\,dx.
\end{equation}
Multiplying \eqref{p1} by $(\alpha+\delta)$, we get
\begin{equation}\label{p4}
(\alpha+\delta)\int_{\Om}(H(\nabla u)^p+\mu(x)H(\nabla u)^q)\,dx=(\alpha+\delta)\int_{\Om}a(x)|u|^{1-\delta}\,dx+\lambda(\alpha+\delta)\int_{\Om}b(x)|u|^{\alpha+1}\,dx.
\end{equation}
Subtracting \eqref{p3} and \eqref{p4}, we get
\begin{equation}\label{p5}
(\alpha+1-p)\int_{\Om}H(\nabla u)^p\,dx+(\alpha+1-q)\int_{\Om}\mu(x)H(\nabla u)^q\,dx=(\alpha+\delta)\int_{\Om}a(x)|u|^{1-\delta}\,dx.
\end{equation}
As in the proof of \cite[Proposition $3.3$]{win} from \eqref{p5} we have
\begin{equation}\label{p6}
\|u\|\leq c_1,
\end{equation}
for some constant $c_1>0$. Moreover,  as in the proof of \cite[Proposition $3.3$]{win}, from \eqref{p3} we get
\begin{equation}\label{p7}
\|u\|\leq\left(\frac{1}{\lambda c_2}\right)^\frac{1}{\alpha+1-p}\quad\text{ or }\quad\|u\|\leq\left(\frac{1}{\lambda c_2}\right)^\frac{1}{\alpha+1-q}.
\end{equation}
Letting $\lambda\to 0$ in \eqref{p7}, we get $\|u\|\to\infty$, which contradicts \eqref{p6}. Hence the proof.
\end{proof} }

For notational convenience, we set
\[B := \int_{\Om}b(x)|u|^{\alpha+1}\,dx\;\text{ and }\; A:=\int_{\Om}a(x)|u|^{1-\delta}\,dx.\]
Defining 
\[\sigma_u(t)= \int_\Om \Big({t^{p-1-\alpha}H(\nabla u)^p}+\mu(x){t^{q-1-\alpha}H(\nabla u)^q}\Big)\,dx- \textcolor{red}{t^{-\alpha-\delta}}A, \]
we see that {$\phi_u^\prime (t)= t^\alpha (\sigma_u(t) -\la B)$.} Thus 
\[tu \in M_\la \;\text{if and only if }\; \sigma_u(t) =\la B.\]
Because $q-1-\alpha>p-1-\alpha >-\delta-\alpha$, we assert that 
\[\lim_{t\to 0^+}\sigma_u(t)=-\infty, \; \lim_{t\to +\infty}\sigma_u(t) = 0, \; \sigma_u(t)>0\;\text{for suitably large} \; t>0.\]
Consider 
\begin{align*}
   &0=\sigma_u^\prime(t) =(p-1-\alpha)t^{p-2-\alpha}\int_\Om H(\nabla u)^p~dx+(q-1-\alpha)t^{q-2-\alpha}\int_\Om \mu(x)H(\nabla u)^q~dx \\
   &\quad \quad+ (\alpha+\delta)t^{-\delta-\alpha-1}A\\
   &\implies E_u(t):= (-p+1+\alpha)t^{p-1+\delta}\int_\Om H(\nabla u)^p~dx +(-q+1+\alpha)t^{q-1+\delta}\int_\Om \mu(x)H(\nabla u)^q~dx= (\alpha+\delta)A
\end{align*}
Keeping into account $0<\delta<1<p<q<\alpha+1$, we observe the following about $E_u$:
\begin{enumerate}
    \item[(i)] $E_u^\prime(t)>0,$ for all $t>0$.
    \item[(ii)] $\lim\limits_{t\to 0^+}E_u(t)=0$.
    \item[(iii)] $ \lim\limits_{t\to \infty} E_u(t)=+\infty$.
\end{enumerate}
From $(ii)$ and $(iii)$, we know that there exists a $t_{\max}>0$ such that $E_u(t_{\max})= (\alpha+\delta)A$. From $(i)$, we agree that this $t_{\max}>0$ must be unique since $E_u$ is an injective map.
Hence this gives that there exists a unique $t_{\max}>0$ such that
\[\sigma_u^\prime(t_{\max})=0 \;\text{and}\; \sigma_u \;\text{increases in} \; (0,t_{\max}).\]
Since for all $t>0$, define another map 
\[\tilde{\sigma}_u(t)= t^{p-1-\alpha}\int_\Om H(\nabla u)^p~dx - t^{-\alpha-\delta}A\]
then 
\begin{align*}
   \tilde{\sigma}_u^\prime(t)  =(p-1-\alpha)t^{p-2-\alpha}\int_\Om H(\nabla u)^p + (\alpha+\delta)t^{-\delta-\alpha-1}A
\end{align*}
which implies $\max\limits_{t>0}\tilde{\sigma}_u(t)= \tilde{\sigma}_u(t_0)$, where
\[t_0:= \left(\frac{(\alpha+\delta)A}{(\alpha-p+1)\int_\Om H(\nabla u)^p~dx}\right)^{\frac{1}{p-1+\delta}}.\]
Clearly $\sigma_u(t)\geq \tilde{\sigma}_u(t)$, therefore we get 
\begin{align*}
    \sigma_u(t_{\max})&\geq \sigma_u(t_0)\geq \tilde{\sigma}_u(t_0)= t_0^{p-1-\alpha}\int_\Om H(\nabla u)^p -t_0^{-\delta-\alpha}A\\
    & = \left(\frac{\alpha-p+1}{\alpha+\delta}\right)^{\frac{\alpha+\delta}{p-1+\delta}}\left(\frac{\delta+p-1}{\alpha-p+1}\right) \frac{\left(\int_\Om H(\nabla u)^p~dx\right)^{\frac{\alpha+\delta}{p-1+\delta}}}{A^{\frac{\alpha+1-p}{p-1+\delta}}}.
\end{align*}
Also we have the following-
\begin{align*}
    A &\leq \left(\int_\Om |a(x)|^{\frac{p^*}{p^*-1+\delta}}~dx\right)^{\frac{p^*+\delta-1}{p^*}} \left(\int_\Om |u|^{p^*}~dx\right)^{\frac{1-\delta}{p^*}}= |a|_{\frac{p^*}{p^*-1+\delta}}|u|_{p^*}^{1-\delta}\\
    B &\leq \left(\int_\Om |b(x)|^{\frac{p^*}{p^*-\alpha-1}}~dx\right)^{\frac{p^*-\alpha-1}{p^*}} \left(\int_\Om |u|^{p^*}~dx\right)^{\frac{\alpha+1}{p^*}}= |b|_{\frac{p^*}{p^*-\alpha-1}}|u|_{p^*}^{\alpha+1}
\end{align*}
Using all these, we obtain that 
\begin{align*}
    \sigma_u(t_{\max})-\la B & \geq \left(\frac{\alpha-p+1}{\alpha+\delta}\right)^{\frac{\alpha+\delta}{p-1+\delta}}\left(\frac{\delta+p-1}{\alpha-p+1}\right) \frac{\left(\int_\Om H(\nabla u)^p~dx\right)^{\frac{\alpha+\delta}{p-1+\delta}}}{A^{\frac{\alpha+1-p}{p-1+\delta}}} - \la B\\
    & \geq \left(\frac{\alpha-p+1}{\alpha+\delta}\right)^{\frac{\alpha+\delta}{p-1+\delta}}\left(\frac{\delta+p-1}{\alpha-p+1}\right)  \frac{\left(cS|u|_{p^*}^p\right)^{\frac{\alpha+\delta}{p-1+\delta}}}{\left( |a|_{\frac{p^*}{p^*-1+\delta}}|u|_{p^*}^{1-\delta} \right)^{\frac{\alpha+1-p}{p-1+\delta}}}-\la |b|_{\frac{p^*}{p^*-\alpha-1}}|u|_{p^*}^{\alpha+1}\\
    &= (D_1-\la D_2)|u|_{p^*}^{\alpha+1}
\end{align*}
where $D_1,D_2>0$ are constants independent of $u$. Clearly, this asserts that there exists a $\la_*>0$ such that  
\begin{equation}\label{fm-1}
    \sigma_u(t_{\max})- \la B>0,\;\; \forall \la \in (0,\la_*).
\end{equation}

\begin{Lemma}\label{t-1,t-2}
Let $\la \in (0,\la_*)$ then for each $u\in  W_0^{1,\mathcal{H}}(\Om)$, the following holds-
\begin{enumerate}
    \item if $B< 0$ then there exists a unique $t_1<t_{\max}$ such that $t_1u\in M_\la^+$ and $I_\la(t_1u)= \inf_{t>0}I_\la(tu)$;
    \item if $B\geq 0$ then there exist unique $t_1,t_2$ satisfying $t_1<t_{\max}<t_2$ such that $t_1u\in M_\la^+$ and $t_2u\in M_\la^-$. 
\end{enumerate}
\end{Lemma}
\begin{proof}
\begin{enumerate}
    \item If $B<0$ then the possible graph of $\sigma_u(t)$ and \eqref{fm-1} suggests that there exists exactly one point $t_1\in (0,t_{\max})$ such that $\sigma_u(t_1)= \la B$ and $\sigma_u^\prime (t_1)>0$ that is \[t_1u\in M_\la.\]
    Now since 
    \[\sigma_u^{\prime}(t) =t^{-\alpha} \phi_u^{\prime \prime}(t)-\alpha t^{-\alpha-1}\phi_u^\prime(t), \]
    we get that $t^{-\alpha} \phi_u^{\prime \prime}(t_1)= \sigma_u^\prime (t_1)>0$. Therefore $t_1u \in M_\la^+$. By uniqueness and property of $t_1$, we infer that $I_\la(t_1u)= \inf_{t>0}I_\la(tu)$.
    \item If $B\geq 0$ then possible graph of $\sigma_u(t)$ and \eqref{fm-1} suggests that there exists exactly two points $0<t_1<t_{\max}<t_2$ such that $\sigma_u(t_1)= \la B=\sigma_u(t_2)$ and $\sigma_u^\prime (t_1)>0>\sigma_u^\prime(t_2)$ that is \[t_1u,t_2u\in M_\la.\]
    Similarly as above case, we get that $t_1u\in M_\la^+$ and $t_2u\in M_\la^-$.
\end{enumerate}
\end{proof}
{
\begin{Lemma}\label{M+}
Let $M_{\lambda}^{+}\neq\emptyset$, then $\inf_{M_{\lambda}^{+}}I_{\lambda}<0$.
\end{Lemma}
\begin{proof}
Let $u\in M_{\lambda}^{+}$, then since $\phi_u^{'}(1)=0$, we get
\begin{equation}\label{q1}
-\frac{1}{1-\delta}\int_{\Om}a(x)|u|^{1-\delta}\,dx=-\frac{1}{1-\delta}\int_{\Om}(H(\nabla u)^p+\mu(x)H(\nabla u)^q)\,dx+\frac{\lambda}{1-\delta}\int_{\Om}b(x)|u|^{\alpha+1}\,dx.
\end{equation}
Also, due to $\phi_u^{\prime\prime}(1)>0$, we get
\begin{equation}\label{q2}
\lambda\int_{\Om}b(x)|u|^{\alpha+1}\,dx<\frac{p+\delta-1}{\alpha+\delta}\int_{\Om}H(\nabla u)^p\,dx+\frac{q+\delta-1}{\alpha+\delta}\int_{\Om}\mu(x)H(\nabla u)^q\,dx.
\end{equation}
Using \eqref{q1}, \eqref{q2} and the fact that $p<q<\alpha+1$, we have
\begin{equation}\label{q3}
\begin{split}
I_{\lambda}(u)&=\left(\frac{1}{p}-\frac{1}{1-\delta}\right)\int_{\Om}H(\nabla u)^p\,dx+\left(\frac{1}{q}-\frac{1}{1-\delta}\right)\int_{\Om}\mu(x)H(\nabla u)^q\,dx\\
&\quad+\lambda\left(\frac{1}{1-\delta}-\frac{1}{\alpha+1}\right)\int_{\Om}b(x)|u|^{\alpha+1}\,dx\\
&\leq\frac{(p+\delta-1)(p-\alpha-1)}{p(1-\delta)(\alpha+1)}\int_{\Om}H(\nabla u)^p\,dx+\frac{(q+\delta-1)(q-\alpha-1)}{q(1-\delta)(\alpha+1)}\int_{\Om}\mu(x)H(\nabla u)^q\,dx<0.
\end{split}
\end{equation}
Hence the result follows.
\end{proof} }

\subsection{Subcritical case}
{
\begin{Lemma}\label{mt-1}
There exists a $\Lambda_*>0$ and a a.e. nonnegative function $u_0$ in $\Om$ such that for any $\la \in (0,\Lambda_*)$, we have $u_0 \in M_\la^+$ and 
\[m_\la^+:=\inf_{M_\la^+} I_\la= I_\la(u_0)<0.\]
\end{Lemma}
\begin{proof}
We fix our $\Lambda_* := \min \{\la_0,\la_*\}$, where $\la_0$ and $\la_*$ are given by Lemma \ref{Nest0emp} and Lemma \ref{t-1,t-2} respectively. Recall that by Lemma \ref{t-1,t-2}, we know that $M_\la^+\neq \emptyset$. So we assume $\{u_n\}_{n\in\mathbb{N}}\subset M_\la^+$ denotes the minimizing sequence for $m_\la^+$, that is 
\[I_\la(u_n)\to m_\la^+<0\;\text{as}\; n\to \infty.\]
Note that $m_\la^+<0$ follows from Lemma \ref{M+} above. By Lemma \ref{cbb}, the sequence $\{u_n\}_{n\in\mathbb{N}}$ is bounded in $W^{1,\mc H}_0(\Om)$ which implies that there exists $u_0\in W^{1,\mc H}_0(\Om)$ such that upto a subsequence
\begin{equation}
    u_n \rightharpoonup u_0 \; \text{weakly in}\; W^{1,\mc H}_0(\Om)\;\text{ and }\; u_n \to u_0 \;\text{strongly in} \; L^r(\Om),\;\text{for every}\;r\in (p,p^*).
\end{equation}
Also we have $b(x)u_n(x) \to b(x)u_0(x)$ pointwise for a.e. $x\in\Om$ as $n\to \infty$. {Since $\{u_n\}_{n\in\mathbb{N}}$ is bounded in $W^{1,\mc H}_0(\Om)$, there exists a constant $K>0$ (independent of $n$) such that $|u_n|_{p^*}^{\alpha+1}\leq K$ for all $n\in \mb N$. This implies for any measurable set $E\subset \Om$, 
\[ \int_{E}b(x)|u_n|^{\alpha+1}\,dx  \leq M\left(\int_E |b(x)|^{\frac{p^*}{p^*-\alpha-1}}~dx\right)^{\frac{p^*-\alpha-1}{p^*}} \]
and since $b\in L^{\frac{p^*}{p^*-\alpha-1}}(\Om)$, for any $\epsilon>0$, there exists a $\eta>0$ such that for any measurable set $E\subset \Om$
such that $|E|<\eta$, 
\[\int_E |b(x)|^{\frac{p^*}{p^*-\alpha-1}}~dx \leq \frac{\epsilon^{\frac{p^*}{p^*-\alpha-1}}}{M}\]
asserting that 
\[\int_{E}b(x)|u_n|^{\alpha+1}\,dx  \leq \epsilon,\;\forall n\in \mb N.\]
 Therefore
\begin{equation}\label{mt-1-1}
   \lim_{n\to \infty}  B_n :=\lim_{n\to \infty} \int_{\Om}b(x)|u_n|^{\alpha+1}\,dx  = \int_{\Om}b(x)|u_0|^{\alpha+1}\,dx
\end{equation}
which follows by applying the Vitali's
convergence theorem. By a similar argument, we obtain
\begin{equation}\label{mt-1-2}
   \lim_{n\to \infty} \int_{\Om}a(x)|u_n|^{1-\delta}\,dx = \int_{\Om}a(x)|u_0|^{1-\delta}\,dx.
\end{equation} }
{Using \eqref{mt-1-1} and \eqref{mt-1-2}} along with weak lower semicontinuity of norms, it follows that
\[I_\la(u_0)\leq \liminf_{n\to \infty}I_\la(u_n) = m_\la^+<0=I_\la(0),\]
which clearly shows that $u_0$ is not identically zero inside $\Om$. {Lemma \ref{t-1,t-2} says that there exists a $t_1>0$ such that $t_1u_0\in M_\la^+$.
By definition of the modular function, we know 
$\liminf_{n\to \infty}\rho_{\mc H}(u_n)\geq  \rho_{\mc H}(u_0)$.}
We now claim that 
\begin{equation}\label{mt-1-3}
    \liminf_{n\to \infty}\rho_{\mc H}(u_n)= \rho_{\mc H}(u_0).
\end{equation}
Because if not, then 
$\liminf\limits_{n\to \infty}\rho_{\mc H}(u_n)> \rho_{\mc H}(u_0)$
which together with \eqref{mt-1-1} and \eqref{mt-1-2} implies that
\begin{align*}
    \liminf_{n\to \infty} \phi_{u_n}^\prime(t_1) &= \liminf_{n\to \infty} \int_{\Om}\Big({t_1^{p-1}H(\nabla u_n)^p}+\mu(x){t_1^{q-1}H(\nabla u_n)^q}\Big)\,dx-{t_1^{-\delta}}\int_{\Om}a(x)|u_n|^{1-\delta}\,dx\\
    &\quad -{\lambda t_1^{\alpha}}\int_{\Om}b(x)|u_n|^{\alpha+1}\,dx\\
    & > \int_{\Om}\Big({t_1^{p-1}H(\nabla u_0)^p}+\mu(x){t_1^{q-1}H(\nabla u_0)^q}\Big)\,dx-{t_1^{-\delta}}\int_{\Om}a(x)|u_0|^{1-\delta}\,dx\\
    &\quad -{\lambda t_1^{\alpha}}\int_{\Om}b(x)|u_0|^{\alpha+1}\,dx= \phi_{u_0}^\prime(t_1)=0.
\end{align*}
Thus, for $n$ sufficiently large, $\phi_{u_n}^\prime(t_1)>0$. But $u_n\in M_\la^+$ and { $\phi_{u_n}^\prime(t)=t^\alpha(\sigma_{u_n}(t)-\la B_n)$  says that $\phi_{u_n}^\prime(t)<0$ for $t\in (0,1)$ and $\phi^\prime_{u_n}(1)=0$. Therefore $t_1>1$ and $\phi_{u_0}$ is decreasing in $(0,t_1]$,} which finally gives that
\[m_\la^+\leq {I_\la(t_1u_0)\leq I_\la(u_0)< \liminf_{n\to \infty}I_\la(u_n)} =m_\la^+. \]
{Thus we arrive at a contradiction, hence proving our claim.} Our next claim is that, upto a subsequence, $u_n \to u_0$ strongly in $W_0^{1,\mc H}(\Om)$ as $n\to \infty$. But this naturally follows from \cite[Proposition 2.2-(v)]{win} and \eqref{mt-1-3}.
Therefore, 
\[m_\la^+ = \lim_{n\to \infty}I_\la(u_n)= I_\la(u_0)\;\text{ and }\; \phi_{u_0}^\prime(1)=0.\]
Also, $\phi_{u_0}^{\prime\prime}(1)= \lim\limits_{n\to \infty}\phi_{u_n}^{\prime\prime}(1)\geq 0$, since $u_n\in M_\la^+$ for each $n\in \mb N$. {But our choice of $\Lambda_*$  and Lemma \ref{t-1,t-2} asserts that $\phi_{u_0}^{\prime\prime}(1)>0$} that is $u_0\in M_\la^+ $. Since $I_\la(u_0)=I_\la(|u_0|)$, we take into account $|u_0|$ instead of $u_0$ and w.l.o.g. assume that $u_0\geq 0$ a.e. in $\Om$. This completes the proof.
\end{proof} }

{{
\begin{Lemma}\label{lem3.5}
Let $u\in M_{\la}^{\pm}$, then there exists $\epsilon>0$ and a continuous function $v:B_{\epsilon}(0)\to (0,\infty)$ such that
\begin{equation}\label{3.5eqn}
v(0)=1\text{ and }v(y)(u+y)\in M_{\la}^{\pm},\quad\forall y\in B_{\epsilon}(0),
\end{equation}
where $B_{\epsilon}(0):=\big\{u\in W_0^{1,\mathcal{H}}(\Om):\|u\|_{W_0^{1,\mathcal{H}}(\Om)}<\epsilon\big\}$.
\end{Lemma}
\begin{proof}
Let us define $F:W_0^{1,\mathcal{H}}(\Om)\times(0,\infty)\to\mathbb{R}$ by 
\begin{align*}
F(y,t)&=t^{p+\delta-1}\int_{\Om}H(u+y)^p\,dx+t^{q+\delta-1}\int_{\Om}\mu(x)H(u+y)^q\,dx\\
&\quad\quad-\int_{\Om}a(x)|u+y|^{1-\delta}\,dx-\la t^{\alpha+\delta}\int_{\Om}b(x)|u+y|^{\alpha+1}\,dx
\end{align*}
for every $y\in W_0^{1,\mathcal{H}}(\Om)$. Since $u\in M_{\la}^{+}\subset M_\la$, we get $F(0,1)=0$. Moreover, due to the fact that $u\in M_{\la}^{+}$, we notice that $F_t((0,1))>0$. Therefore, by the implicit function theorem, there exists $\epsilon>0$ and a continuous function $v:B_{\epsilon}(0)\to(0,\infty)$ such that
{
$$
v(0)=1\text{ and }F(y,v(y))=0\quad\forall \;y\in B_{\epsilon}(0)
$$
i.e. $F(y,v(y)) \in M_\la$ when $y\in B_\epsilon(0)$. A quick observation suggests that 
\[F_t(y,v(y))= (v(y))^\delta \phi_{u+y}^{\prime \prime}(v(y)),\;\;\;\; \forall\; y\in B_\epsilon(0)\]
is continuous. Therefore since $F_t((0,1))>0$, we can choose $\epsilon>0$ small enough further to get
$$
v(0)=1\text{ and }v(y)(u+y)\in M_\la^{+}\quad\forall\; y\in B_{\epsilon}(0).
$$}
The proof for $M_{\la}^{-}$ is similar.
\end{proof}

\begin{Lemma}\label{p3.6}
Let $h\in W_0^{1,\mathcal{H}}(\Om)$ and  $\Lambda_*,u_0$ are as given by Lemma \ref{mt-1}. Then for any $\la\in(0,\Lambda_*)$, there exists $\beta>0$ such that $I_\la(u_0)\leq I_\la(u_0+th)$ for every $t\in[0,\beta]$. 
\end{Lemma}
\begin{proof}
Let us define $g_h:[0,\infty)\to\mathbb{R}$ by
\begin{align*}
g_h(t)&=(p-1)\int_{\Om}H(\nabla (u_0+th))^p\,dx+(q-1)\int_{\Om}\mu(x)H(\nabla (u_0+th))^q\,dx\\
&\quad\quad+\delta\int_{\Om}a(x)|u_0 +th|^{1-\delta}\,dx-\la\alpha{\int_{\Om}b(x)|u_0 +th|^{\alpha+1}\,dx} = \phi_{u_0+th}^{\prime \prime}(1).
\end{align*}
Using the fact that $\phi_{u_0}^{''}(1)>0$  and 
{continuity of $g_h$, there exists ${\beta_1}>0$ such that $g_h(t)>0$ for all $t\in[0,{\beta_1}]$.  On the other hand, by Lemma \ref{lem3.5}, for every $t\in[0,{\beta_1}]$, we find a $\epsilon>0$ such that there exist a continuous map $v: B_\epsilon(0) \to (0,\infty)$ such that
\begin{equation}\label{p363}
v(th)(u_0 +th)\in M_\la^{+},\;\forall \;t\in \{t\in[0,\beta_1]:\; th\in B_\epsilon(0)\}\;\text{ and }v(th)\to 1\text{ as }t\to 0^{+}.
\end{equation}
The continuity of map $u_0\mapsto \phi_{u_0}^{''}(1)$ and $\phi_{u_0+th}'(v(th))=0$ suggests that there exists $\beta \in (0,\beta_1)$ such that $\phi_{u_0+th}^{\prime\prime}(1)>0$ and $\phi_{u_0+th}(v(th)) \leq \phi_{u_0+th}(1)$ when $t\in [0,\beta]$.
 By Lemma \ref{mt-1}, we have 
 \begin{equation}\label{p364}
 m_\la^{+}=I_\la(u_0)\leq I_\la\big(v(th)(u_0 +th)\big)= \phi_{u_0+th}(v(th))\leq \phi_{u_0+th}(1) = I_\la(u_0+th), \;\forall\; t\in[0,{\beta}]. 
 \end{equation}
 }
Hence the result follows.
\end{proof} }

\begin{Lemma}\label{p37}
Suppose $\Lambda_*$ and $u_0\in M_\la^{+}$ are as given by Lemma \ref{mt-1}. Then for any $\lambda\in(0,\Lambda_*)$, $u_0$ is a weak solution of \eqref{maineqn} such that $I_\la(u_0)<0$.
\end{Lemma}
\begin{proof}
Let $h\in W_0^{1,\mathcal{H}}(\Om)$ be nonnegative in $\Om$. Then by Lemma \ref{p3.6} for every $t\in(0,\beta)$, we have
$
I_\la(u_0+th)-I_\la(u_0)\geq 0.
$
Therefore,
\begin{equation}\label{p371}
\begin{split}
&\frac{1}{p}\int_{\Om}(H(\nabla(u_0 + th))^p-H(\nabla u_0)^p)\,dx+
\frac{1}{q}\int_{\Om}\mu(x)(H(\nabla(u_0 +th))^q-H(\nabla u_0)^q)\,dx\\
&\quad-\frac{\la}{\alpha+1}\int_{\Om}b(x)(|u_0+th|^{\alpha+1}-|u_0|^{\alpha+1})\,dx-\frac{1}{1-\delta}\int_{\Om}a(x)(|u_0 + th|^{1-\delta}-|u_0|^{1-\delta})\,dx\geq 0.
\end{split}
\end{equation}
Dividing both sides of \eqref{p371} by $t$ and letting $t\to 0^+$, we obtain
\begin{equation}\label{p372}
I-J-K\geq 0,
\end{equation}
where
\begin{equation}\label{iest}
\begin{split}
I&=\lim_{t\to 0}\left(\frac{1}{p}\int_{\Om}\frac{H(\nabla(u_0 + th))^p-H(\nabla u_0)^p}{t}\,dx+
\frac{1}{q}\int_{\Om}\mu(x)\frac{H(\nabla(u_0 +th))^q-H(\nabla u_0)^q}{t}\,dx\right)\\
&=\int_{\Om}\{H(\nabla u_0)^{p-1}\nabla_{\eta}H(\nabla u_0)+\mu(x)H(\nabla u_0)^{q-1}\nabla_{\eta}H(\nabla u_0)\}\nabla h\,dx,
\end{split}
\end{equation}
\begin{equation}\label{jest}
\begin{split}
J&=\lim_{t\to 0}\frac{\la}{\alpha+1}\int_{\Om}b(x)\frac{|u_0+th|^{\alpha+1}-|u_0|^{\alpha+1}}{t}\,dx=\frac{\la}{\alpha+1}\int_{\Om}b(x)u_0^{\alpha}h\,dx
\end{split}
\end{equation}
and
\begin{align*}
K&=\lim_{t\to 0}\frac{1}{1-\delta}\int_{\Om}a(x)\frac{|u_0 + th|^{1-\delta}-|u_0|^{1-\delta}}{t}\,dx=\int_{\Om}a(x)(u_0 +\xi  h)^{-\delta}h\,dx,
\end{align*}
where $\xi\to 0$ as $t\to 0^+$. Since $a(x)(u_0 +\xi t h)^{-\delta} h\geq 0$ in $\Om$, by Fatou's lemma it follows that
\begin{equation}\label{kest}
K\geq \int_{\Om}a(x)u_0 ^{-\delta}h\,dx.
\end{equation}
Using the estimates \eqref{iest}, \eqref{jest} and \eqref{kest} in \eqref{p372} we get 
\begin{equation}\label{p373}
\begin{split}
0&\leq\int_{\Om}\{H(\nabla u_0)^{p-1}\nabla_{\eta}H(\nabla u_0)+\mu(x)H(\nabla u_0)^{q-1}\nabla_{\eta}H(\nabla u_0)\}\nabla h\,dx\\
&\quad-\int_{\Om}(\la b(x)u_0^{\alpha}+a(x)u_0^{-\delta})h\,dx.
\end{split}
\end{equation}
Let $\phi\in W_0^{1,\mathcal{H}}(\Om)$ and we define $\psi=(u_0+\epsilon\phi)^+=\max\{u_0+\epsilon\phi,0\}$ for $\epsilon>0$. Note that $\psi\in W_0^{1,\mathcal{H}}(\Om)$ is nonnegative in $\Om$. Therefore, by putting $h=\psi$ in \eqref{p373} we obtain
\begin{equation}\label{p374}
\begin{split}
0&\leq\int_{\Om}\{H(\nabla u_0)^{p-1}\nabla_{\eta}H(\nabla u_0)+\mu(x)H(\nabla u_0)^{q-1}\nabla_{\eta}H(\nabla u_0)\}\nabla((u_0+\epsilon \phi)^+)\,dx\\
&\quad-\int_{\Om}(\la b(x)u_0^{\alpha}+a(x)u_0^{-\delta})(u_0+t \phi)^+\,dx\\
&=\int_{\{u_0 +\epsilon\phi\geq 0\}}\{H(\nabla u_0)^{p-1}\nabla_{\eta}H(\nabla u_0)+\mu(x)H(\nabla u_0)^{q-1}\nabla_{\eta}H(\nabla u_0)\}\nabla((u_0+\epsilon \phi))\,dx\\
&\quad-\int_{\{u_0 +\epsilon\phi\geq 0\}}(\la b(x)u_0^{\alpha}+a(x)u_0^{-\delta})(u_0+t \phi)\,dx\\
&=\int_{\Om}\{H(\nabla u_0)^{p-1}\nabla_{\eta}H(\nabla u_0)+\mu(x)H(\nabla u_0)^{q-1}\nabla_{\eta}H(\nabla u_0)\}\nabla((u_0+\epsilon \phi))\,dx\\
&\quad-\int_{\{u_0 +\epsilon\phi<0\}}\{H(\nabla u_0)^{p-1}\nabla_{\eta}H(\nabla u_0)+\mu(x)H(\nabla u_0)^{q-1}\nabla_{\eta}H(\nabla u_0)\}\nabla((u_0+\epsilon \phi))\,dx\\
&\quad-\int_{\Omega}(\la b(x)u_0^{\alpha}+a(x)u_0^{-\delta})(u_0+\epsilon \phi)\,dx+\int_{\{u_0 +\epsilon\phi<0\}}(\la b(x)u_0^{\alpha}+a(x)u_0^{-\delta})(u_0+\epsilon \phi)\,dx\\
&=\int_{\Om}(H(\nabla u_0)^p+\mu(x)H(\nabla u_0)^q-a(x)u_0^{1-\delta}-\la b(x)u_0^{\alpha+1})\,dx\\
&\quad+\epsilon\int_{\Om}\{H(\nabla u_0)^{p-1}\nabla_{\eta}H(\nabla u_0)+\mu(x)H(\nabla u_0)^{q-1}\nabla_{\eta}H(\nabla u_0)\}\nabla \phi\,dx\\
&\quad-\epsilon\int_{\Om}(\la b(x)u_0^{\alpha}+a(x)u_0^{-\delta})\phi\,dx\\
&\quad-\int_{\{u_0 +\epsilon\phi<0\}}\{H(\nabla u_0)^{p-1}\nabla_{\eta}H(\nabla u_0)+\mu(x)H(\nabla u_0)^{q-1}\nabla_{\eta}H(\nabla u_0)\}\nabla((u_0+\epsilon \phi))\,dx\\
&\quad+\int_{\{u_0 +\epsilon\phi<0\}}(\la b(x)u_0^{\alpha}+a(x)u_0^{-\delta})(u_0+\epsilon \phi)\,dx
\end{split}
\end{equation}
\begin{equation*}
\begin{split}
&\leq \epsilon\int_{\Om}\{H(\nabla u_0)^{p-1}\nabla_{\eta}H(\nabla u_0)+\mu(x)H(\nabla u_0)^{q-1}\nabla_{\eta}H(\nabla u_0)\}\nabla \phi\,dx\\
&\quad-\epsilon\int_{\Om}(\la b(x)u_0^{\alpha}+a(x)u_0^{-\delta})\phi\,dx\\
&\quad-\epsilon\int_{\{u_0 +\epsilon\phi<0\}}\{H(\nabla u_0)^{p-1}\nabla_{\eta}H(\nabla u_0)+\mu(x)H(\nabla u_0)^{q-1}\nabla_{\eta}H(\nabla u_0)\}\nabla\phi\,dx\\
&\quad+\la\int_{\{u_0 +\epsilon\phi<0\}}b(x)u_0^{\alpha}(u_0+\epsilon \phi)\,dx,
\end{split}
\end{equation*}
where we have used that
\begin{equation}\label{estI1p37}
\begin{split}
\int_{\Om}(H(\nabla u_0)^p+\mu(x)H(\nabla u_0)^q-a(x)u_0^{1-\delta}-\la b(x)u_0^{\alpha+1})\,dx=0,
\end{split}
\end{equation}
which holds, since $u_0\in M_\la^{+}\subset M_\la$. Moreover, we have used the nonnegativity of $a(x)$ and of $u_0$ (by Lemma \ref{mt-1}) along with
\begin{equation}\label{estI4p37}
\begin{split}
\int_{\{u_0 +\epsilon\phi<0\}}(H(\nabla u_0)^p+\mu(x)H(\nabla u_0)^q)\,dx\geq 0.
\end{split}
\end{equation}
We observe that, since $|\{u_0 +\epsilon\phi<0\}|\to 0$ as $\epsilon\to 0^{+}$, it follows that
\begin{equation}\label{estI5p37}
\lim_{\epsilon\to 0}\int_{\{u_0 +\epsilon\phi<0\}}\{H(\nabla u_0)^{p-1}\nabla_{\eta}H(\nabla u_0)+\mu(x)H(\nabla u_0)^{q-1}\nabla_{\eta}H(\nabla u_0)\}\nabla \phi\,dx=0.
\end{equation}
Now we estimate the last integral in the right hand side of \eqref{p374}, say
\[
I=\la\int_{\{u_0 +\epsilon\phi<0\}}b(x)u_0^{\alpha}(u_0+\epsilon \phi)\,dx.
\]
Notice that if $u_0+\epsilon\phi<0$, then $u_0<-\epsilon\phi$ and therefore, we have
\begin{equation}\label{estI7p37}
\begin{split}
\frac{1}{\epsilon}|I|&\leq\frac{\la}{\epsilon}\int_{\{u_0 +\epsilon\phi<0\}}|b(x)u_0^{\alpha+1}+\epsilon b(x)u_0^{\alpha}\phi|\,dx 
\leq
C\la\epsilon^{\alpha}\|b\|_{L^m(\Om)}\|\phi\|^{\alpha+1}
\end{split}
\end{equation}
for some positive constant $C$, independent of $\epsilon$. Therefore, we have
\begin{equation}\label{estI7final}
\lim_{\epsilon\to 0}\frac{1}{\epsilon}|I|=0.
\end{equation}
Combining the above estimates \eqref{estI5p37} and \eqref{estI7final} in
\eqref{p373}, we arrive at
\begin{equation}\label{p37f}
\begin{split}
0&\leq\int_{\Om}\{H(\nabla u_0)^{p-1}\nabla_{\eta}H(\nabla u_0)+\mu(x)H(\nabla u_0)^{q-1}\nabla_{\eta}H(\nabla u_0)\}\nabla \phi\,dx\\
&\quad-\int_{\Om}(\la b(x)u_0^{\alpha}+a(x)u_0^{-\delta})\phi\,dx.
\end{split}
\end{equation}
Since $\phi\in W_0^{1,\mathcal{H}}(\Om)$ is arbitrary, the equality holds in \eqref{p37f} and thus $u_0$ is a weak solution of \eqref{maineqn}. By Lemma \ref{mt-1}, we have $I_\la(u_0)<0$.

Recall from Lemma \ref{mt-1} that $u_0\geq 0$ a.e. in $\Omega$. Finally, we prove that $u_0>0$ a.e. in $\Omega$. Here we follow the arguments in the proof of \cite[Page $16$, Proposition $3.6$]{Patpos}. To this end, suppose there exists $D\subset\Omega$ such that $|D|>0$ and $u=0$ in $D$. Let $h\in W_0^{1,\mathcal{H}}(\Omega)$ be such that $h>0$ a.e. in $\Omega$. Then by Lemma \ref{p3.6} for every $t\in(0,\beta)$, we have $(u_0 +th)^{1-\delta}>u_0^{1-\delta}$ in $\Omega\setminus D$. Using this fact, we obtain for every $t\in(0,\beta)$ that
\begin{equation}\label{pos1}
\begin{split}
0&\leq\frac{I_{\la}(u_0 +th)-I_{\la}(u_0)}{t}\\
&=\frac{1}{tp}\int_{\Om}(H(\nabla(u_0 + th))^p-H(\nabla u_0)^p)\,dx+
\frac{1}{tq}\int_{\Om}\mu(x)(H(\nabla(u_0 +th))^q-H(\nabla u_0)^q)\,dx\\
&\quad-\frac{\la}{t(\alpha+1)}\int_{\Om}b(x)(|u_0+th|^{\alpha+1}-|u_0|^{\alpha+1})\,dx-\frac{1}{t(1-\delta)}\int_{\Om}a(x)((u_0 + th)^{1-\delta}-{u_0}^{1-\delta})\,dx\\
&=\frac{1}{tp}\int_{\Om}(H(\nabla(u_0 + th))^p-H(\nabla u_0)^p)\,dx+
\frac{1}{tq}\int_{\Om}\mu(x)(H(\nabla(u_0 +th))^q-H(\nabla u_0)^q)\,dx\\
&\quad-\frac{\la}{t(\alpha+1)}\int_{\Om}b(x)(|u_0+th|^{\alpha+1}-|u_0|^{\alpha+1})\,dx-\frac{1}{1-\delta}\Big\{\int_{\Omega\setminus{D}}a(x)\frac{(u_0 + th)^{1-\delta}-{u_0}^{1-\delta}}{t}\,dx\\
&\quad+t^{-\delta}\int_{D}a(x)h^{1-\delta}\,dx\Big\}\\
&\leq \frac{1}{tp}\int_{\Om}(H(\nabla(u_0 + th))^p-H(\nabla u_0)^p)\,dx+
\frac{1}{tq}\int_{\Om}\mu(x)(H(\nabla(u_0 +th))^q-H(\nabla u_0)^q)\,dx\\
&\quad-\frac{\la}{t(\alpha+1)}\int_{\Om}b(x)(|u_0+th|^{\alpha+1}-|u_0|^{\alpha+1})\,dx-\frac{t^{-\delta}}{1-\delta}\int_{D}a(x)h^{1-\delta}\,dx.\\
\end{split}
\end{equation}
Thus, we arrive at
\begin{equation}\label{pos2}
\begin{split}
0&\leq\frac{I_{\la}(u_0 +th)-I_{\la}(u_0)}{t}\to -\infty\text{ as }t\to 0^+,
\end{split}
\end{equation}
which is a contradiction. Hence, $u_0>0$ a.e. in $\Om$. This completes the proof.
\end{proof}
}

\begin{Lemma}\label{M-neg}
There exists $\Lambda_0>0$ such that $\inf_{M_\la^-}I_\la> 0$ whenever $\la \in(0,\Lambda_0)$.
\end{Lemma}
\begin{proof}
{Lemma \ref{t-1,t-2} says that $M_\la^-\neq \emptyset$ for $\la \in (0,\la_*)$, so let $u\in M_\la^-$ and $\la\in (0,\lambda_*)$. On contrary, we also assume that $u\in M_\la^-$ is such that $I_\la(u)\leq 0$. } Since $\phi_u^\prime(1)=0$, we get 
\begin{align}\label{M-neg-1}
    \int_\Om H(\nabla u)^p~dx \leq \frac{pq}{q-p} \left[\frac{q-1+\delta}{q(1-\delta)}\int_\Om a(x)|u|^{1-\delta}~dx +\la \frac{\alpha+1-q}{q(\alpha+1)}\int_\Om b(x)|u|^{\alpha +1}~dx\right].
\end{align}
Now using $\phi_u^\prime(1)=0$ and $\phi_u^{\prime \prime}(1)<0$, we obtain
\begin{align*}
    &\int_\Om (p-1+\delta)H(\nabla u)^p~dx+(q-1+\delta)\int_\Om \mu(x)H(\nabla u)^q~dx\\
    &< \la (\alpha+\delta)\int_\Om b(x)|u|^{\alpha+1}~dx \leq \la(\alpha+\delta) |b|_{\frac{p^*}{p^*-\alpha-1}}|u|_{p^*}^{\alpha+1}\\
    & \leq \la(\alpha+\delta) |b|_{\frac{p^*}{p^*-\alpha-1}} \left(S\int_\Om |\nabla u|^p~dx\right)^{\frac{\alpha+1}{p}} \\
    &\leq \la(\alpha+\delta) |b|_{\frac{p^*}{p^*-\alpha-1}} (c_2S)^{\frac{\alpha+1}{p}}\left(\int_\Om H(\nabla u)^p~dx\right)^{\frac{\alpha+1}{p}}.
\end{align*}
This gives 
\begin{align}\label{M-neg-2}
    \int_\Om H(\nabla u)^p~dx > \left(\frac{p-1+\delta}{\la (\alpha+\delta)(c_2S)^{\frac{\alpha+1}{p}}  |b|_{\frac{p^*}{p^*-\alpha-1}} }\right)^{\frac{p}{\alpha+1-p}}.
\end{align}
Using \eqref{M-neg-2} in \eqref{M-neg-1}, we get 
\begin{align*}
   \frac{1}{\la^{\frac{\alpha+1}{\alpha+1-p}}}\left(\frac{p-1+\delta}{ (\alpha+\delta)(c_2S)^{\frac{\alpha+1}{p}}  |b|_{\frac{p^*}{p^*-\alpha-1}} }\right)^{\frac{p}{\alpha+1-p}}&< \frac{pq}{q-p} \left[\frac{q-1+\delta}{\la q(1-\delta)}\int_\Om a(x)|u|^{1-\delta}~dx\right.\\
   &\quad \quad \left.+ \frac{\alpha+1-q}{q(\alpha+1)}\int_\Om b(x)|u|^{\alpha +1}~dx\right].
\end{align*}
Since $\alpha+1>p$, passing on the limits $\lambda \to 0^+$ in above, we arrive at 
\[+\infty\leq \frac{p(\alpha+1-q)}{(q-p)(\alpha+1)}\int_\Om b(x)|u|^{\alpha +1}~dx,\]
which is absurd. Hence there must exist a $\Lambda_0\in (0, \la_*]$ such that $\inf_{M_\la^-}I_\la\geq 0$ when $\la \in (0,\Lambda_0).$
\end{proof}


\begin{Lemma}\label{mt-2}
For any $\la \in (0,\Lambda_0)$, there exists a $v_0 \in M_\la^-$ such that $v_0\geq 0 $ a.e. in $\Om$ and 
\[0<m_\la^-:=\inf_{M_\la^-} I_\la= I_\la(v_0).\]
\end{Lemma}
\begin{proof}
Let us consider $\{v_n\}_{n\in\mathbb{N}}\subset M_\la^-$ as a minimizing sequence for $m_\la^-$, that is 
\[I_\la(v_n)\to m_\la^->0\;\text{as}\; n\to \infty.\]
By Lemma \ref{cbb}, we get that $\{v_n\}_{n\in\mathbb{N}}$ is bounded in $W^{1,\mc H}_0(\Om)$ which implies that there exists a $v_0\in W^{1,\mc H}_0(\Omega)$ such that upto a subsequence
\begin{equation}
    v_n \rightharpoonup v_0 \; \text{in}\; W^{1,\mc H}_0(\Om)\;\text{and}\; v_n \to v_0 \;\text{in} \; L^r(\Om),\;\text{for}\;r\in (p,p^*).
\end{equation}
Following proof of Lemma \ref{mt-1}, we will get $v_0\neq 0$ map over $\Om$, so Lemma \ref{t-1,t-2} will give us that there exists a $t_2>0$ such that $t_2 v_0\in M_\la^-$. Now, following the exact arguments in the proof of Lemma \ref{mt-1}, we obtain the existence of $v_0 \in M_\la^-$ such that $v_0\geq 0$ a.e. in $\Om$ and 
\[0<m_\la^-:=\inf_{M_\la^-} I_\la= I_\la(v_0).\]
\end{proof}

\begin{Lemma}\label{sec-weak-sol}
Let $\lambda\in(0,\Lambda_0)$ and $v_0\in M_\la^{-}$ be as given by Lemma \ref{mt-2}. Then $v_0$ is a weak solution of \eqref{maineqn} such that $I_\la(v_0)>0$.
\end{Lemma}
\begin{proof}
We follow similar arguments as in Lemma \ref{p37}. Using Lemma \ref{lem3.5} and reproducing same arguments as in Lemma \ref{p3.6}, we assert that there exists a $\epsilon>0$ and a continuous map $\zeta: B_\epsilon(0)\to(0,
\infty)$ such that
\[\zeta(th)\to 1 \;\text{as}\; t\to 0^+,\;\; \zeta(th)(v_0+th)\in M_\la^-\;\text{and}\; I_\la(v_0)\leq I_{\la}(\zeta(th)(v_0+th))\]
for all $t$ such that $th\in B_\epsilon(0)$. Now one can replicate the arguments used in Lemma \ref{p37} to prove that $v_0$ is a weak solution of \eqref{maineqn} whereas it remains to show that $v_0>0$ a.e. in $\Om$ { which can be proved in a similar way as we proved $u_0>0$ a.e. in $\Omega$.}
\end{proof}
\textbf{Proof of Theorem \ref{scthm}:} Let $(\Lambda_*,u_0)$ and $(\Lambda_0,v_0)$ are given by Lemma \ref{p37} and Lemma \ref{sec-weak-sol} respectively. We define $\Lambda=\min\{\Lambda_*,\Lambda_0\}$, then $u_0$ is a negative energy solution of \eqref{maineqn} by Lemma \ref{p37} and $v_0$ is a positive energy solution of \eqref{maineqn} by Lemma \ref{sec-weak-sol} respectively. 

\subsection{Critical case}
We consider \eqref{maineqn} with the following assumptions on $a$ and $b$- {Here $\alpha+1=p^*$.}
\begin{enumerate}
    \item[(H3)] $a>0$ a.e. in $\Om$ and $a\in L^{\frac{p^*}{p^*-1+\delta}}(\Om)$,
    \item[(H4)] $b$ is possibly sign-changing a.e. in $\Om$ with $b^+\not\equiv 0$ and $b\in L^\infty(\Om)$.
\end{enumerate}
Define the best constant as follows 
\[S_1= \inf_{u\in W^{1,p}_0(\Om)\setminus \{0\}} \frac{|\nabla u|_p^p}{|u|_{p^*}^p}.\]
{Let us fix the constants $c,C>0$ which satisfies the following
\begin{equation}\label{new}
    c\int_\Om H(\nabla u)^p~dx \leq \int_\Om |\nabla u|^p~dx \leq C\int_\Om H(\nabla u)^p~dx. 
\end{equation}}
\begin{Proposition}\label{gap}
There exists a $\Lambda^*>0$ such that $M_\la$ possesses a gap structure when $\la \in (0,\Lambda^*)$ i.e. for some $D_1,D_2>0$,
\[\int_\Om H(\nabla U)^p~dx \geq D_1>D_2 \geq \int_\Om( H(\nabla u)^p+\mu(x) H(\nabla u)^q)~dx,\;\forall \; u\in M_\la^+,\; U\in M_\la^-.\]
\end{Proposition}
\begin{proof}
Let $u\in M_\la^+$ then $\phi_u^\prime(1)=0$ and $\phi_u^{\prime\prime}(1)>0$ gives us that 
\begin{align*}
    &\int_{\Om} \left((p^*-p)H(\nabla u)^p + (p^*-q){\color{red}\mu(x)}H(\nabla u)^q\right)~dx  < (p^*-1+\delta)\int_\Om a(x) |u|^{1-\delta}~dx\\
    & < (p^*-1+\delta) |a|_{\frac{p^*}{p^*-1+\delta}} |u|_{p^*}^{1-\delta} < (p^*-1+\delta) |a|_{\frac{p^*}{p^*-1+\delta}} S_{1}^{-\frac{1-\delta}{p}}|\nabla u|_p^{1-\delta}\\
    & \leq (p^*-1+\delta) |a|_{\frac{p^*}{p^*-1+\delta}} (S_{1}^{-1}C)^{\frac{1-\delta}{p}}\left(\int_\Om H(\nabla u)^p~dx\right)^{\frac{1-\delta}{p}}, \; \text{using}\; \eqref{new}.
\end{align*}
From this, it follows that 
\begin{equation}\label{gap-1}
    \int_\Om H(\nabla u)^p~dx< \left[\frac{p^*-1+\delta}{p^*-p} |a|_{\frac{p^*}{p^*-1+\delta}} (S_{1}^{-1}C)^{\frac{1-\delta}{p}}\right]^{\frac{p}{p-1+\delta}}:=A_1(say).
\end{equation}
Putting this in the later equation, we obtain
\begin{align}
     &\int_\Om \mu(x) H(\nabla u)^q~dx\nonumber\\ \label{gap-2}
     &<  \left[\frac{p^*-1+\delta}{p^*-q} |a|_{\frac{p^*}{p^*-1+\delta}} (S_1^{-1}C)^{\frac{1-\delta}{p}}\right] \left[\frac{p^*-1+\delta}{p^*-p} |a|_{\frac{p^*}{p^*-1+\delta}} (S_1^{-1}C)^{\frac{1-\delta}{p}}\right]^{\frac{1-\delta}{p-1+\delta}}\\ \nonumber
     &=\left[(p^*-1+\delta) |a|_{\frac{p^*}{p^*-1+\delta}} (S_1^{-1}C)^{\frac{1-\delta}{p}}\right]^{\frac{p}{p-1+\delta}}\frac{1}{(p^*-q)(p^*-p)^{\frac{1-\delta}{p-1+\delta}}}:=A_2(say)\nonumber
\end{align}
Combining \eqref{gap-1} and \eqref{gap-2} and taking $D_2=A_1+A_2$, we get 
\[\int_\Om( H(\nabla u)^p+\mu(x) H(\nabla u)^q)~dx\leq D_2,\; \forall \; u\in M_\la^+.\]
Now let $U\in M_\la^-$ then $\phi_U^\prime(1)=0$ and $\phi_U^{\prime\prime}(1)<0$ gives us that 
\begin{align*}
    &\int_{\Om} \left((p-1+\delta)H(\nabla U)^p + (q-1+\delta){\color{red}\mu(x)}H(\nabla U)^q\right)~dx\\
    &< \la(p^*-1+\delta)\int_\Om b(x) |U|^{p^*}~dx < \la (p^*-1+\delta) |b|_\infty |U|_{p^*}^{p^*}\\
    & < \la (p^*-1+\delta) |b|_\infty S_1^{-\frac{p^*}{p}}|\nabla U|_p^{p^*} \leq \la (p^*-1+\delta) |b|_\infty (S_1^{-1}C)^{\frac{p^*}{p}} \left(\int_\Om H(\nabla U)^p~dx\right)^{\frac{p^*}{p}}.
\end{align*}
This easily suggests that 
\begin{equation}
    \int_\Om H(\nabla U)^p~dx \geq \left(\frac{(S_1^{-1}C)^{\frac{p^*}{p}}}{\la |b|_{\infty}}\left(\frac{p-1+\delta}{p^*-1+\delta}\right)\right)^{\frac{p}{p^*-p}}:=D_1 (say)
\end{equation}
and hence 
\[\int_\Om H(\nabla U)^p~dx \geq D_1,\;\forall \; U\in M_\la^-.\]
Observing that $D_1$ depends on $\la $, if we choose 
\[0<\la< \Lambda^* := \frac{\frac{(S_1^{-1}C)^{\frac{p^*}{p}}}{ |b|_{\infty}}\left(\frac{p-1+\delta}{p^*-1+\delta}\right)}{D_2^{\frac{p^*-p}{p}}}\]
then $D_1>D_2$ and this completes the proof.
\end{proof}

\begin{Corollary}
When $\la \in (0,\Lambda^*)$ then $M_\la^-$ is closed in $W^{1,\mc H}_0(\Om)$ topology.
\end{Corollary}
\begin{proof}
Let $\{u_k\}_{{\color{red}k\in\mathbb{N}}}$ be a sequence in $M_\la^-$ which converges to $u_0$ strongly in $W^{1,\mc H}_0(\Om)$. Then $u_0\in \overline{M_\la^-}= M_\la^-\cup M_\la^0$. By Proposition \ref{gap}, we get the following
\[0<D_1\leq \lim_{k\to \infty} \int_\Om( H(\nabla u_k)^p+\mu(x) H(\nabla u_k)^q)~dx =\lim_{k\to\infty}\rho_{\mc H}(\nabla u_k) = \rho_{\mc H}(\nabla u_0). \]
{This implies that $u_0\in M_\la^-$ and our proof is complete.}
\end{proof}

Now that we have $M_\la^+\cup \{0\}$ and $M_\la^-$ as closed subsets of $W^{1,\mc H}_0(\Om)$, we apply the Ekeland variational principle to get a sequence $\{u_k\}_{{\color{red}k\in\mathbb{N}}}\subset Z^\pm$ as a minimizing sequence for $I_\la$ in $Z^\pm$, where 
\[Z^+ = M_\la^+\cup \{0\} \;\text{and}\; Z^-= M_\la^-.\]
The sequence $\{u_k\}_{{\color{red}k\in\mathbb{N}}}$ satisfies the following-
\begin{enumerate}
    \item[(I)] $m_\la^\pm \leq I_\la \leq m_\la^\pm +\frac{1}{k}$;
    \item[(II)] $I_\la(v) \geq I_\la(u_k) + \frac{\|v-u_k\|}{k}$, for any $v\in Z^\pm$,
\end{enumerate}
{where $m_\lambda^+= \inf\limits_{M_\lambda^+}I_\lambda$ and $m_\lambda^-= \inf\limits_{M_\lambda^-}I_\lambda$.} By Lemma \ref{cbb}, we know that $\{u_k\}_{{\color{red}k\in\mathbb{N}}}$ must be bounded in $M_\la$ and thus, up to a subsequence, 
\[u_k\rightharpoonup u_0\;\text{as}\; k \to \infty\]
for some $u_0 \in W^{1,\mc H}_0(\Om)$. By Lemma \ref{M+}, we know $m_\la^+<0$. So $ I_\la(u_k) \to m_\la^+<0$ says that $u_0 \not\equiv 0$. By virtue of $(H2)$, we say that $I_\la(u_k)=I_\la(|u_k|)$, so we may w.l.o.g. assume that $u_k\geq 0$ and $u_0\geq 0$ a.e. in $\Om$.

\begin{Lemma}\label{liminf-pos-I}
Let $\la\in (0,\min\{\la_*, \Lambda^*\})$ and $\{u_k\}_{{\color{red}k\in\mathbb{N}}}\subset M_\la^+$, then 
\begin{align*}
&\liminf_{k\to \infty}\left[ \int_\Om \left((p+\delta-1)H(\nabla u_k)^p+(q+\delta-1)\mu(x)H(\nabla u_k)^q\right)~dx\right.\\
&\quad \left.-\la (p^*+\delta-1)\int_\Om b(x)|u_k|^{p^*}~dx\right]>0.
\end{align*}
\end{Lemma}
\begin{proof}
Since $\{u_k\}\subset M_\la^+$, we have $\phi_{u_k}^\prime(1) =0$ and $\phi_{u_k}^{\prime \prime}(1)>0$. This gives rise to two {equivalent} inequalities-
\begin{enumerate}
    \item[(a)]  $\int_\Om \left((p+\delta-1)H(\nabla u_k)^p+(q+\delta-1)\mu(x)H(\nabla u_k)^q\right)~dx -\la (p^*+\delta-1)\int_\Om b(x)|u_k|^{p^*}~dx>0$;
    \item[(b)] $\int_\Om \left((p-p^*)H(\nabla u_k)^p+(q-p^*)\mu(x)H(\nabla u_k)^q\right)~dx + (p^*+\delta-1)\int_\Om a(x)|u_k|^{1-\delta}~dx>0$.
\end{enumerate}
{So to prove this Lemma, it is enough to show that 
\[\liminf\limits_{k\to \infty}\left[\int_\Om \left((p-p^*)H(\nabla u_k)^p+(q-p^*)\mu(x)H(\nabla u_k)^q\right)~dx + (p^*+\delta-1)\int_\Om a(x)|u_k|^{1-\delta}~dx\right]>0.\]}
Due to (b), we already have 
\[\liminf\limits_{k\to \infty}\left[\int_\Om \left((p-p^*)H(\nabla u_k)^p+(q-p^*)\mu(x)H(\nabla u_k)^q\right)~dx + (p^*+\delta-1)\int_\Om a(x)|u_k|^{1-\delta}~dx\right]\geq 0\]
and on contrary, let us assume that 
\begin{equation}\label{liminf-pos-1}
    \liminf\limits_{k\to \infty}\left[\int_\Om \left((p-p^*)H(\nabla u_k)^p+(q-p^*)\mu(x)H(\nabla u_k)^q\right)~dx + (p^*+\delta-1)\int_\Om a(x)|u_k|^{1-\delta}~dx\right]=0.
\end{equation}
Recalling \eqref{mt-1-2} and using it in \eqref{liminf-pos-1}, we obtain 
\begin{equation}\label{liminf-pos-2}
    \liminf\limits_{k\to \infty}\left[\int_\Om \left((p^*-p)H(\nabla u_k)^p+(p^*-q)\mu(x)H(\nabla u_k)^q\right)~dx \right] =(p^*+\delta-1)\int_\Om a(x)|u_0|^{1-\delta}~dx.
\end{equation}
By virtue of \eqref{liminf-pos-2}, we can assume that there exist {$E_1,\;E_2\geq0$, but not both zero at the same time,} such that 
\[\lim_{k\to \infty}\int_\Om H(\nabla u_k)^p~dx=E_1\;\text{and}\; \lim_{k\to \infty}\int_\Om \mu(x) H(\nabla u_k)^q~dx=E_2.\]
Therefore, \eqref{liminf-pos-2} gives that 
\begin{equation}\label{liminf-pos-3}
     (p^*+\delta-1)\int_\Om a(x)|u_0|^{1-\delta}~dx= \left(\frac{p^*-p}{p^*+\delta-1}\right)E_1 + \left(\frac{p^*-q}{p^*+\delta-1}\right)E_2.
\end{equation}
From $\lim\limits_{k\to \infty}\phi_{u_k}^{\prime}(1)=0$, we have
\[E_1+E_2 -\int_\Om a(x)|u_0|^{1-\delta}~dx = \la \lim_{k\to \infty}b(x)|u_k|^{p^*}~dx.\]
Using \eqref{liminf-pos-2} in above, we obtain
\begin{equation}\label{liminf-pos-4}
    \la \lim_{k\to \infty}b(x)|u_k|^{p^*}~dx= \left(\frac{p+\delta-1}{p^*+\delta-1}\right)E_1 + \left(\frac{q+\delta-1}{p^*+\delta-1}\right)E_2.
\end{equation}
Recalling \eqref{fm-1} and the fact that $\la\in(0,\la_*)$, we assert that $\sigma_{u_k}(t_{\max})-\la \int_\Om b(x)|u_k|^{p^*}~dx>0$. Passing on the limits $k\to \infty$ in this, we get 
\[\left(\frac{p^*-p}{p^*+\delta-1}\right)^{\frac{p^*+\delta-1}{p-1+\delta}}\left(\frac{p+\delta-1}{p^*-p}\right) {E_1^{\frac{p^*+\delta-1}{p+\delta-1}}}{\left(\int_\Om a(x)|u_0|^{1-\delta}~dx\right)^{\frac{p-p^*}{p-1+\delta}}} - \la\lim_{k\to \infty}b(x)|u_k|^{p^*}~dx \geq 0. \]
Now {if $E_2>0$ then} using \eqref{liminf-pos-3} in the above and simplifying the terms, we shall get 
\[-E_2 \left(\frac{q+\delta-1}{p^*+\delta-1}\right)\geq 0\]
which is a contradiction, since $E_2,\; q+\delta-1,\;p^*+\delta-1>0$. This renders us the desired result. {If $E_2=0$ then using \eqref{liminf-pos-1} and \eqref{liminf-pos-2}, we obtain
\[\left(\frac{p^*-p}{p^*+\delta-1}\right)^{\frac{p^*+\delta-1}{p-1+\delta}}\left(\frac{p+\delta-1}{p^*-p}\right) {E_1^{\frac{p^*+\delta-1}{p+\delta-1}}}{\left(\frac{(p^*-p)E_1}{p+\delta-1}\right)^{\frac{p-p^*}{p-1+\delta}}} - \left(\frac{p+\delta-1}{p^*+\delta-1}\right)E_1 \geq 0\]
which on simplification gives
\[\left(\frac{p+\delta-1}{p^*+\delta-1}\right)E_1\left(\left(\frac{p+\delta-1}{p^*+\delta-1}\right)^{\frac{p^*-p}{p+\delta-1}}-1\right)\geq 0,\]
which is a contradiction, since $E_2=0$ implies $E_1>0$. This completes the proof.
} 
\end{proof}

\begin{Lemma}\label{liminf-pos-II}
Let $\la\in (0,\min\{\la_*, \Lambda^*\})$ and $\{u_k\}\subset M_\la^-$ then 
\begin{align*}
    &\liminf_{k\to \infty}\left[ \int_\Om \left((p+\delta-1)H(\nabla u_k)^p+(q+\delta-1)\mu(x)H(\nabla u_k)^q\right)~dx\right.\\
   &\quad  \left.-\la (p^*+\delta-1)\int_\Om b(x)|u_k|^{p^*}~dx\right]<0.
\end{align*}
\end{Lemma}
\begin{proof}
Following exactly the arguments of Lemma \ref{liminf-pos-I}, we can prove this result.
\end{proof}
{By virtue of Lemma \ref{lem3.5}(as the proof still works in the critical case), for each $k\in \mathbb N$, we obtain a sequence of maps $\{\zeta_k\}:B_{\epsilon_k}(0) \to (0,\infty)$, where $\epsilon_k>0$,} such that 
\[\zeta_k(0)=1\;\text{and}\; \zeta_k(th)(u_k+th)\in M_\la^\pm\]
for sufficiently small $t>0$. From $u_k \in M_\la$ and $\zeta_k(th)(u_k+th)\in M_\la$, we have 
\begin{enumerate}
    \item[(G1)] \begin{equation*}
        \int_{\Om}\Big({H(\nabla u_k)^p}+\mu(x){H(\nabla u_k)^q}\Big)\,dx-\int_{\Om}a(x)|u_k|^{1-\delta}\,dx-{\lambda }\int_{\Om}b(x)|u_k|^{p^*}\,dx=0;
    \end{equation*}
    \item[(G2)] \begin{align*}
    &\int_{\Om}\Big({\zeta_k^{p}(th)H(\nabla w_k)^p}+\mu(x){\zeta_k^{q}(th)H(\nabla w_k)^q}\Big)\,dx-{\zeta_k^{1-\delta}(th)}\int_{\Om}a(x)|w_k|^{1-\delta}\,dx\\
    &\quad -{\lambda \zeta_k^{p^*}(th)}\int_{\Om}b(x)|w_k|^{p^*}\,dx=0,
    \end{align*}
    where $w_k :=u_k+th$.
\end{enumerate}
We denote $\langle \zeta_k^\prime(0),h\rangle $ as the dual action on $W^{1,\mc H}_0(\Om)$.

\begin{Lemma}\label{unif-bdd}
When $\la \in (0,\min\{\la_*, \Lambda^*\})$, $\langle \zeta_k^\prime(0),h\rangle $ is uniformly bounded w.r.t. $k$, for any  non negative $h\in W^{1,\mc H}_0(\Om)$.
\end{Lemma}
\begin{proof}
Let us first assume the case that $\{u_k\}\subset M_\la^+$. Then subtracting (G1) from (G2), we get
\begin{equation}\label{div}
\begin{split}
    & \int_\Om \left(\left(\zeta_k(th)^p-1\right)H(\nabla w_k)^p+ \mu(x) \left(\zeta_k(th)^q-1\right)H(\nabla w_k)^q\right)dx \\
    &\quad + \int_\Om \left((H(\nabla w_k)^p-H(\nabla u_k)^p)+\mu(x)(H(\nabla w_k)^q-H(\nabla u_k)^q)\right)dx\\
    & \quad \quad - \left(\zeta_k^{1-\delta}(th)-1\right)\int_\Om a(x)|w_k|^{1-\delta}~dx -\int_\Om a(x) \left(|w_k|^{1-\delta}-|u_k|^{1-\delta}\right)dx\\
    & \quad \quad \quad - \la \left(\zeta_k^{p^*}(th)-1\right)\int_\Om b(x)|w_k|^{p^*}~dx -\la \int_\Om b(x) \left(|w_k|^{p^*}-|u_k|^{p^*}\right)dx =0.
\end{split}
\end{equation}
Taking into account $\int_\Om a(x) \left(|w_k|^{1-\delta}-|u_k|^{1-\delta}\right)dx\geq 0$ since $h$ is non negative, dividing both sides of the equation \eqref{div} by $t>0$ and then passing the limit $t\to 0^+$, we obtain
\begin{equation*}
\begin{split}
    0&\leq \langle \zeta_k^\prime(0),h\rangle \left(\int_\Om \left(pH(\nabla u_k)^p+q\mu(x) H(\nabla u_k)^q\right)dx -(1-\delta) \int_\Om a(x)|u_k|^{1-\delta}dx \right.\\
    &\left.\quad -\la p^* \int_\Om b(x)|u_k|^{p^*}dx\right)+ \int_\Om \left(pH^{p-1}(\nabla u_k)+q \mu(x) H^{q-1}(\nabla u_k)\right)\nabla_\eta H(\nabla u_k)\nabla h~dx\\
    &\quad \quad -\la p^* \int_\Om b(x) |u_k|^{p^*-1}h~dx.
\end{split}
\end{equation*}
Since $\phi_{u_k}^\prime(1)=0$, from above inequality we obtain
\begin{equation}\label{div1}
\begin{split}
    0&\leq \langle \zeta_k^\prime(0),h\rangle \left(\int_\Om \left((p+\delta-1)H(\nabla u_k)^p+ (q+\delta-1)\mu(x) H(\nabla u_k)^q\right)dx\right.\\
    &\quad \left.-\la (p^*+\delta-1) \int_\Om b(x)|u_k|^{p^*}dx\right)\\
    &\quad + \int_\Om \left(pH^{p-1}(\nabla u_k)+ q\mu(x) H^{q-1}(\nabla u_k)\right)\nabla_\eta H(\nabla u_k)\nabla h~dx -\la p^* \int_\Om b(x) u_k^{p^*-1}h~dx.
    \end{split}
    \end{equation}
    Applying Lemma \ref{liminf-pos-I} in \eqref{div1} and boundedness of $\{u_k\}_{k\in \mathbb N}$ in $W^{1,\mc H}_0(\Om)$, we deduce that $\langle \zeta_k^\prime(0),h\rangle$ is bounded from below, for all non negative  $h\in W^{1,\mc H}_0(\Om)$. Now it remains to show that $\langle \zeta_k^\prime(0),h\rangle$ is bounded from above. Suppose, on contrary,
    {$\liminf_{k\to \infty}\langle \zeta_k^\prime(0),h\rangle=+\infty.$}
    {Then  for sufficiently large $k$, $\zeta_k(th)>1=\zeta_k(0)$. Recalling the fact that $u_k,\; \zeta_k(th)w_k \in M_\la^+$ and putting $v=\zeta_k(th)w_k $ in (II) (see the part before Lemma \ref{liminf-pos-I}), we obtain
    \begin{align*}
        &\frac{\|u_k\|}{k}(\zeta_k(th)-1)+\frac{t\|h\|}{k}\zeta_k(th)  \geq \frac{\|\zeta_k(th)w_k-u_k\|}{k} \geq I_\la(u_k) - I_\la(\zeta_k(th)w_k)\\
        & = \left(\frac{1}{p}-\frac{1}{1-\delta}\right)\int_\Om \left(H(\nabla u_k)^p-H( \zeta_k(th)\nabla w_k)^p\right)dx \\
        &\quad + \left(\frac{1}{q}-\frac{1}{1-\delta}\right)\int_\Om \mu(x)\left(H(\nabla u_k)^q-H( \zeta_k(th)\nabla w_k)^q\right)dx\\
        &\quad\quad +\la \left(\frac{1}{1-\delta}-\frac{1}{p^*}\right)\int_\Om b(x) \left(|u_k|^{p^*}-|\zeta_k(th)w_k|^{p^*}\right)dx\\
        &= \left(\frac{1}{1-\delta}-\frac{1}{p}\right)\int_\Om \left(H(\nabla w_k)^p-H( u_k)^p\right)dx + \left(\frac{1}{1-\delta}-\frac{1}{p}\right)(\zeta_k(th)^p-1)\int_\Om H(\nabla w_k)^pdx\\
        & \quad +\left(\frac{1}{1-\delta}-\frac{1}{q}\right)\int_\Om \mu(x) \left(H(\nabla w_k)^q-H( u_k)^q\right)dx\\
        &\quad \quad +  \left(\frac{1}{1-\delta}-\frac{1}{q}\right)(\zeta_k(th)^q-1)\int_\Om \mu(x) H(\nabla w_k)^q dx \\
         &\quad \quad \quad -\la \left(\frac{1}{1-\delta}-\frac{1}{p^*}\right)\int_\Om b(x)(|w_k|^{p^*}-|u_k|^{p^*})dx\\
        &\quad \quad \quad \quad - \la \left(\frac{1}{1-\delta}-\frac{1}{p^*}\right)(\zeta_k^{p^*}(th)-1)\int_\Om b(x)|w_k|^{p^*}dx.
    \end{align*}
    Diving both sides of the above inequality with $t>0$ and passing on the limit $t\to 0^+$, we get 
    \begin{align*}
        &\langle \zeta_k^\prime(0),h \rangle \frac{\|u_k\|}{k}+\frac{\|h\|}{k}  \geq \frac{\langle \zeta_k^\prime(0),h \rangle}{1-\delta} \left(\int_\Om \left((p+\delta-1)H(\nabla u_k)^p + \mu(x) (q+\delta-1)H(\nabla u_k)^q\right)dx \right.\\
        &\quad \left.-\la (p^*+\delta -1)\int_\Om b(x) |u_k|^{p^*}~dx \right)-\la \frac{p^*+\delta-1}{1-\delta}\int_\Om b(x)u_k^{p^*-1}h~dx\\
        &\quad \quad +\frac{1}{1-\delta}\int_\Om \left((p+\delta-1)H^{p-1}(\nabla u_k) + \mu(x)(q+\delta-1)H^{q-1}(\nabla u_k) \right) \nabla_\eta H(\nabla u_k)\nabla h~dx.
    \end{align*}
    Hence we get 
    \begin{align*}
        &\frac{\|h\|}{k}  \geq \frac{\langle \zeta_k^\prime(0),h \rangle}{1-\delta} \left(\int_\Om \left((p+\delta-1)H(\nabla u_k)^p + \mu(x) (q+\delta-1)H(\nabla u_k)^q\right)dx \right.\\
        &\quad \left.-\la (p^*+\delta -1)\int_\Om b(x) |u_k|^{p^*}~dx -\frac{1-\delta}{k}\|u_k\| \right)-\la \frac{p^*+\delta-1}{1-\delta}\int_\Om b(x)u_k^{p^*-1}h~dx\\
        &\quad \quad +\frac{1}{1-\delta}\int_\Om \left((p+\delta-1)H^{p-1}(\nabla u_k) + \mu(x)(q+\delta-1)H^{q-1}(\nabla u_k) \right) \nabla_\eta H(\nabla u_k)\nabla h~dx
    \end{align*}
    which will become absurd if $\liminf\limits_{k\to \infty}\langle \zeta_k^\prime(0),h \rangle = +\infty$, since $\|u_k\|$ is bounded and Lemma \ref{liminf-pos-I} holds true. So $\langle \zeta_k^\prime(0),h\rangle$ is bounded above as well. Therefore, $\langle \zeta_k^\prime(0),h\rangle $ is uniformly bounded w.r.t. $k$, for any  non negative $h\in W^{1,\mc H}_0(\Om)$ when $\{u_k\}\subset M_\la^+$. An exactly same argument will lead to establish $\langle \zeta_k^\prime(0),h\rangle $ is uniformly bounded w.r.t. $k$, for any  non negative $h\in W^{1,\mc H}_0(\Om)$ when $\{u_k\}\subset M_\la^-$, using Lemma \ref{liminf-pos-II}.}
 \end{proof}

\begin{Lemma}\label{weak-solI}
For any $h\in W^{1,\mc H}_0(\Om)$ and $k\in \mb N$, 
$\displaystyle\int_\Om a(x) u_k^{-\delta}h~dx <\infty$.
Moreover, the following holds true-
\begin{equation}\label{div3}
\begin{split}
   &\lim_{k\to \infty} \int_{\Om}\left(H(\nabla u_k)^{p-1}+\mu(x)H(\nabla u_k)^{q-1}\right)\nabla_{\eta}H(\nabla u_k)\nabla h\,dx\\
   &\quad {-\int_{\Om}\left(a(x)u_k^{-\delta}+\la b(x)u_k^{p^*-1}\right)h\,dx}=0.
\end{split}
\end{equation}
\end{Lemma}
\begin{proof}
{Let us first assume that $h\in W^{1,\mc H}_0(\Om)$ and $h\geq 0$  in $\Omega$.} Then following same ideas as in proof of Lemma \ref{unif-bdd}, we write the following
\begin{align*}
    &\frac{\|u_k\|}{k}(\zeta_k(th)-1)+\frac{t\|h\|}{k}\zeta_k(th)   \geq I_\la(u_k) - I_\la(\zeta_k(th)w_k)\\
    & =-\frac{\zeta_k(th)^p-1}{p}\int_\Om H(\nabla u_k)^pdx-\frac{\zeta_k(th)^p}{p}\int_\Om \left(H(\nabla w_k)^p-H(\nabla u_k)^p\right)~dx\\
    &\quad -\frac{\zeta_k(th)^q-1}{q}\int_\Om \mu(x) H(\nabla u_k)^qdx-\frac{\zeta_k(th)^q}{q}\int_\Om \mu(x) \left(H(\nabla w_k)^q-H(\nabla u_k)^q\right)~dx\\
    & \quad \quad+ \frac{\zeta_k^{1-\delta}(th)-1}{1-\delta}\int_\Om a(x)|w_k|^{1-\delta}dx+\frac{1}{1-\delta}\int_\Om a(x)\left(|w_k|^{1-\delta}-|u_k|^{1-\delta}\right)~dx\\
    &\quad \quad \quad+ \la \frac{\zeta_k^{p^*}(th)-1}{p^*}\int_\Om b(x)|w_k|^{p^*}dx +\frac{\la}{p^*}\int_\Om b(x)\left(|w_k|^{p^*}-|u_k|^{p^*}\right)~dx.
\end{align*}
Diving both sides of the above equation by $t>0$ and then passing the limit $t\to 0^+$, we get 
\begin{align}
   & \langle \zeta_k^\prime(0),h\rangle \frac{\|u_k\|}{k}+ \frac{\|h\|}{k} \nonumber\\
   & \geq -\langle \zeta^\prime(0),h\rangle \phi_{u_k}^\prime(1) - \int_\Om \left(H^{p-1}(\nabla u_k)+\mu(x)H^{q-1}(\nabla u_k)\right)\nabla_\eta H(\nabla u_k)\nabla h~dx \nonumber\\
   &\quad + \la \int_\Om b(x)u_k^{p^*-1}h~dx +\lim_{t\to 0^+}\frac{1}{1-\delta}\int_\Om \frac{a(x)\left(|w_k|^{1-\delta}-|u_k|^{1-\delta}\right)}{t}dx\nonumber \\
   & \geq - \int_\Om \left(H^{p-1}(\nabla u_k)+\mu(x)H^{q-1}(\nabla u_k)\right)\nabla_\eta H(\nabla u_k)\nabla h~dx \nonumber\\
   & \quad+ \la \int_\Om b(x)u_k^{p^*-1}h~dx + \int_\Om a(x) u_k^{-\delta}h~dx \label{weak-solI-1}
\end{align}
{where we used $a(x)>0$, Fatou's Lemma} and $u_k\in M_\la$ to obtain the last step. From here, employing Lemma \ref{unif-bdd} and boundedness of the sequence $\{u_k\}$ in $W^{1,\mc H}_0(\Om)$ we conclude that 
\[\int_\Om a(x) u_k^{-\delta}h~dx <\infty, \;\forall\;\; 0\leq h\in W^{1,\mc H}_0(\Om)\]
but it is enough to ascertain integrability of '$a(x)u_k^{-\delta}h$', for any $h \in W^{1,\mc H}_0(\Om)$.

Next we prove the estimate \eqref{div3}. To this end, passing limit $k\to \infty$ in \eqref{weak-solI-1} we get 
\begin{align}
     &\int_{\Om}\left(H(\nabla u_k)^{p-1}+\mu(x)H(\nabla u_k)^{q-1}\right)\nabla_{\eta}H(\nabla u_k)\nabla h\,dx\nonumber \\
   &\quad -{\int_{\Om}\left(a(x)u_k^{-\delta}+\la  b(x)u_k^{p^*-1}\right)h\,dx}\geq o_k(1).\label{weak-solI-2}
\end{align}
Now it remains to show that \eqref{weak-solI-2} holds for any arbitrary $h\in W^{1,\mc H}_0(\Om)$. For this purpose,
we test \eqref{weak-solI-2} with $ \psi_\epsilon ^+$, where $\psi_\epsilon= u_k + \epsilon h$, $h\in W^{1,\mc H}_0(\Om)$ is arbitrary and  we obtain 
\begin{align*}
    o_k(1) &\leq \int_{\Om}\left[\left(H(\nabla u_k)^{p-1}+\mu(x)H(\nabla u_k)^{q-1}\right)\nabla_{\eta}H(\nabla u_k)\nabla \psi_\epsilon^+- \left(a(x)u_k^{-\delta}+\la b(x)u_k^{p^*-1}\right)\psi_\epsilon^+\right]dx\\
    &  \leq \int_\Om \left(H(\nabla u_k)^p+\mu(x) H(\nabla u_k)^q\right)dx-\int_\Om \left(a(x)|u_k|^{1-\delta}+\la b(x) |u_k|^{p^*}\right)dx\\
    & \quad + \epsilon\left[\int_{\Om}\left(H(\nabla u_k)^{p-1}+\mu(x)H(\nabla u_k)^{q-1}\right)\nabla_{\eta}H(\nabla u_k)\nabla h\,dx\right.\\
    & \quad \quad \left.-\int_\Om \left(a(x)u_k^{-\delta}+\la b(x)u_k^{p^*-1}\right)h\,dx \right]-\int_\Om \left(a(x)u_k^{-\delta}+\la b(x)u_k^{p^*-1}\right)\psi_\epsilon^-\\
    & \quad \quad\quad  + \int_{\Om}\left(H(\nabla u_k)^{p-1}\nabla_{\eta}H(\nabla u_k)+\mu(x)H(\nabla u_k)^{q-1}\right)\nabla_{\eta}H(\nabla u_k)\nabla \psi_\epsilon^-\,dx\\
    & \leq \epsilon\left[\int_{\Om}\left(H(\nabla u_k)^{p-1}+\mu(x)H(\nabla u_k)^{q-1}\right)\nabla_{\eta}H(\nabla u_k)\nabla h\,dx\right.\\
    & \quad \quad \left.-\int_\Om \left(a(x)u_k^{-\delta}+\la b(x)u_k^{p^*-1}\right)h\,dx \right]-\int_\Om \left(a(x)u_k^{-\delta}+\la b(x)u_k^{p^*-1}\right)\psi_\epsilon^-\\
    & \quad \quad\quad  + \int_{\Om}\left(H(\nabla u_k)^{p-1}+\mu(x)H(\nabla u_k)^{q-1}\right)\nabla_{\eta}H(\nabla u_k)\nabla \psi_\epsilon^-\,dx
\end{align*}
where we used $\psi_\epsilon ^+ = \psi_\epsilon +\psi_\epsilon^-$ and $u_k\in M_\la$. We define $\Om_\epsilon^- = Supp \; \psi_\epsilon^-$ and use it in above inequality to get
\begin{align}
    o_k(1) & \leq \epsilon\left[\int_{\Om}\left(H(\nabla u_k)^{p-1}+\mu(x)H(\nabla u_k)^{q-1}\right)\nabla_{\eta}H(\nabla u_k)\nabla h\,dx\right. \nonumber\\
    & \quad \quad \left.-\int_\Om \left(a(x)u_k^{-\delta}+ \la b(x)u_k^{p^*-1}\right)h\,dx \right]-\int_{\Om_\epsilon^-} \left(a(x)u_k^{-\delta}+ \la b(x)u_k^{p^*-1}\right)\psi_\epsilon^-\nonumber\\
    & \quad \quad\quad  + \int_{\Om_\epsilon^-}\left(H(\nabla u_k)^{p-1}+\mu(x)H(\nabla u_k)^{q-1}\right)\nabla_{\eta}H(\nabla u_k)\nabla \psi_\epsilon^-\,dx\nonumber\\
    & \leq \epsilon\left[\int_{\Om}\left(H(\nabla u_k)^{p-1}+\mu(x)H(\nabla u_k)^{q-1}\right)\nabla_{\eta}H(\nabla u_k)\nabla h\,dx\right.\nonumber\\
    & \quad \quad \left.-\int_\Om \left(a(x)u_k^{-\delta}+\la b(x)u_k^{p^*-1}\right)h\,dx \right] - \la\int_{\Om_\epsilon^-} b(x)u_k^{p^*-1}(u_k+\epsilon h)^-~dx\nonumber\\
    &  \quad+\int_{\Om_\epsilon^-}\left(H(\nabla u_k)^{p-1}+\mu(x)H(\nabla u_k)^{q-1}\right)\nabla_{\eta}H(\nabla u_k)\nabla \psi_\epsilon^-\,dx, \label{weak-solI-3}
\end{align}
where we used $\int_{\Om_\epsilon^-} a(x)u_k^{-\delta}\psi_\epsilon^-~dx>0$. We have the following separate estimates via some easy calculations-
\begin{enumerate}
    \item[(i)] $\int_{\Om_\epsilon^-} H(\nabla u_k)^{p-1}\nabla_{\eta}H(\nabla u_k)\nabla \psi_\epsilon^-\,dx \leq -\epsilon \int_{\Om_\epsilon^-} H(\nabla u_k)^{p-1}\nabla_{\eta}H(\nabla u_k)\nabla h\,dx\leq  \epsilon C \|h\|\|u_k\| $;
    \item[(ii)] $\int_{\Om_\epsilon^-}\mu(x) H(\nabla u_k)^{q-1}\nabla_{\eta}H(\nabla u_k)\nabla \psi_\epsilon^-\,dx \leq -\epsilon \int_{\Om_\epsilon^-} \mu(x)H(\nabla u_k)^{q-1}\nabla_{\eta}H(\nabla u_k)\nabla h\,dx\leq  \epsilon C \|h\|\|u_k\| $;
    \item[(iii)] $\int_{\Om_\epsilon^-} b(x)u_k^{p^*-1}(u_k+\epsilon h)^-~dx \leq \|b\|_\infty \epsilon^{p^*} (\int_{\Om_\epsilon^-} |h|^{p^*}dx) + C\epsilon \|b\|_\infty (\int_{\Om_\epsilon^-}|h|^{p^*})^{\frac{1}{p^*}}$.
\end{enumerate}
Using (i)-(iii) in \eqref{weak-solI-3}, dividing the inequality by $\epsilon$ and passing $\epsilon \to 0^+$, we get
\begin{align*}
   &\lim_{k\to \infty} \int_{\Om}\left(H(\nabla u_k)^{p-1}+\mu(x)H(\nabla u_k)^{q-1}\right)\nabla_{\eta}H(\nabla u_k)\nabla h\,dx\\
   &\quad -\int_{\Om}\left(a(x)u_k^{-\delta}+\la b(x)u_k^{p^*-1}\right)h\,dx=0,\; \text{for any}\; h \in W^{1,\mc H}_0(\Om)
\end{align*}
taking into account $|\Om_\epsilon^-| \to 0$ as $\epsilon \to 0^+$. This completes the proof, since $h$ is arbitrary.
\end{proof}

Define 
\[S_2= \inf_{u\in W^{1,p}_0(\Om)\setminus\{0\}}\frac{\int_\Om H(\nabla u)^p~dx}{\left(\int_\Om|u|^{p^*}~dx \right)^{\frac{p}{p^*}}}.\]

\begin{Proposition}\label{strong-con}
There exists a $\Lambda_{**}>0$ such that when $\la \in (0,{\color{red}\Lambda_{**}})$ and $\lim\limits_{k\to \infty}I_\la(u_k)= m_\la^+$ then there exists $u_0\in W^{1,\mc H}_0(\Om)$ such  $u_k \to u_0$ strongly(up to a subsequence) as $k\to \infty$. Moreover 
\[u_0\in M_\la^+\;\text{and}\; I_\la(u_0)= M_\la^+.\]
\end{Proposition}
\begin{proof}By virtue of Proposition \ref{gap}, we choose 
\[\Lambda_0 = \sup\left\{\la>0:\; \sup\left\{\int_\Om H(\nabla u)^p:\; u\in M_\la^+\right\}\leq \left(\frac{p^*S_{2}^{\frac{p^*}{p}}}{\la |b|_\infty p}\right)^{\frac{p}{p^*-p}}\right\}\]
and set $\Lambda_{**}= \min\{\la_*,\Lambda_0, \Lambda_*\}$. We now fix $\la \in (0,\Lambda_{**})$. By the boundedness {and non negativity of} $\{u_k\}$, we say that there exists non negative $u_0\in W^{1,\mc H}_0(\Om)$ such that 
\begin{align*}
    u_k &\rightharpoonup u_0 \;\text{weakly in }\; W^{1,\mc H}_0(\Om)\;\text{and}\; L^{p^*}(\Om), \\
    u_k & \to u_0\;\text{strongly in}\; L^s(\Om),\;\forall\; s \in [1,p^*),\\
    \|u_k\| & \to U\geq 0.
\end{align*}
If $U=0$ then $u_k \to 0$ strongly in $W^{1,\mc H}_0(\Om)$ which will imply $I_\la\to 0>m_\la^+$, a contradiction. Thus $U>0$ and $u_0$ is a non zero function.
We have 
\begin{equation}\label{weak-solI-10}
    \lim_{k \to \infty}\int_\Om b(x) u_k^{p^*-1}u_0~dx = \int_\Om b(x)u_0^{p^*}
\end{equation}
By Example(b)(Pg.488) of \cite{BLieb83} and $\mu, \;b\in L^\infty(\Om)$, we get the following
\begin{align}
    \int_{\Om}H(\nabla u_k)^p~dx &=   \int_{\Om}H(\nabla (u_k-u_0))^p~dx+  \int_{\Om}H(\nabla u_0)^p~dx+o_k(1)\label{weak-solI-4}\\
    \int_{\Om}\mu(x)H(\nabla u_k)^q~dx &=   \int_{\Om}\mu(x)H(\nabla (u_k-u_0))^q~dx+  \int_{\Om}\mu(x)H(\nabla u_0)^q~dx+o_k(1)\label{weak-solI-5}\\
    \int_\Om b(x) |u_k|^{p^*}~dx & = \int_\Om b(x) |u_k-u_0|^{p^*}~dx+ \int_\Om b(x) |u_0|^{p^*}~dx+o_k(1).\label{weak-solI-6}
\end{align}
Using \eqref{weak-solI-10}-\eqref{weak-solI-6} in Lemma \ref{weak-solI} with $h=(u_k-u_0)$, we get
\begin{align}
    o_k(1) &=  \int_{\Om}\left(H(\nabla u_k)^{p-1}+\mu(x)H(\nabla u_k)^{q-1}\right)\nabla_{\eta}H(\nabla u_k)\nabla (u_k-u_0)\,dx\nonumber\\
   &\quad -\int_{\Om}\left(a(x)u_k^{-\delta}+\la b(x)u_k^{p^*-1}\right)(u_k-u_0)\,dx\label{weak-solI-8}\\
   & = \int_{\Om}\left(H(\nabla (u_k-u_0))^p+\mu(x)H(\nabla (u_k-u_0))^q\right)~dx \nonumber\\
   & \quad -\int_\Om a(x)u_k^{-\delta}(u_k-u_0)~dx - \la \int_\Om b(x)|u_k-u_0|^{p^*}~dx \label{weak-solI-7}.
\end{align}
From \eqref{weak-solI-8}, 
\begin{align}
    &\lim_{k\to \infty}\int_{\Om}\left(H(\nabla (u_k-u_0))^p+\mu(x)H(\nabla (u_k-u_0))^q\right)~dx\nonumber \\
    &\quad = \lim_{k\to \infty}\int_\Om a(x)u_k^{-\delta}(u_k-u_0)~dx +\la  \lim_{k\to \infty}\int_\Om b(x)|u_k-u_0|^{p^*}\nonumber\\
    &:= K_a+\la d^{p^*} \;\text{(say)}.\label{weak-solI-9}
\end{align}
By Lebesgue dominated convergence theorem, we have 
\[\lim_{k\to \infty}\int_\Om a(x)u_k^{1-\delta}~dx= \int_\Om a(x)u_0^{1-\delta}~dx.\]
Lemma \ref{weak-solI} says that $a(x)u_k^{-\delta}u_0 \in L^1(\Om)$ and thus applying Fatou's Lemma, we get 
\[\liminf_{k \to \infty}\int_\Om a(x)u_k^{-\delta}u_0~dx \geq  \int_\Om a(x)u_0^{1-\delta}~dx.\]
Implementing this in \eqref{weak-solI-9} and using \eqref{weak-solI-4} and \eqref{weak-solI-5}, we obtain $K_a\leq 0$ that is
\begin{align}
    &\lim_{k\to \infty} \int_\Om \left(H(\nabla u_k)^p+\mu(x)H(\nabla u_k)^q\right)~dx- \int_\Om \left(H(\nabla u_0)^p+\mu(x)H(\nabla u_0)^q\right)~dx\nonumber\\
    &= \lim_{k\to \infty}\int_{\Om}\left(H(\nabla (u_k-u_0))^p+\mu(x)H(\nabla (u_k-u_0))^q\right)~dx \leq \la d^{p^*}.\label{weak-solI-11}
\end{align}
which says $d\geq 0$. If $d=0$ then we get $u_k\to u_0$ strongly in $W^{1,\mc H}_0(\Om)$ from \eqref{weak-solI-11}, hence concluding the proof. Let $d>0$ and assume
\[\lim_{k\to \infty}\int_\Om H(\nabla w_k)^p~dx=l_1^p,\;  \lim_{k\to \infty}\int_\Om \mu(x) H(\nabla w_k)^q~dx=l_2^q \;\text{and}\; \int_\Om b(x)|w_k|^{p^*}~dx=d^{p^*}>0.\]
Using \eqref{weak-solI-10}-\eqref{weak-solI-6} in $\lim_{k\to \infty}I_\la(u_k)=m_\la^+$, we have 
\begin{align}
    I_\la(u_0) + \int_\Om \left( \frac{H(\nabla w_k)^p}{p}+ \frac{\mu(x) H(\nabla w_k)^q}{q}\right)~dx-\frac{\la}{p^*}\int_\Om b(x)|w_k|^{p^*}~dx=m_\la^++o_k(1)
\end{align}
that is 
\begin{equation}\label{weak-solI-12}
    0>m_\la^+ = I_\la(u_0)+\frac{l_1^p}{p}+\frac{l_2^q}{q}-\frac{\la d^{p^*}}{p^*}.
\end{equation}
Since $u_k\in M_\la^+$, again using \eqref{weak-solI-10}-\eqref{weak-solI-6} we get 
\begin{equation}
    o_k(1) + \int_\Om \left( H(\nabla w_k)^p+\mu(x)H(\nabla w_k)^q\right)~dx -\la \int_\Om b(x) |w_k|^{p^*}~dx + \phi_{u_0}^\prime(1)=0
\end{equation}
that is 
\begin{equation}\label{weak-solI-13}
    l_1^p+l_2^q -\la d^{p^*}+\phi_{u_0}^\prime(1)=0.
\end{equation}
Since $0<\la<\la_*$ and $u_0\not\equiv 0$, there exists $0<t_1<t_2$ such that 
$\phi_{u_0}^\prime(t_1)=\phi_{u_0}^\prime(t_2) =0$ and $t_1u_0\in M_\la^+$, $t_2u_0\in M_\la^-$. We study following three cases:
\begin{enumerate}
    \item[(i)] $t_2<1$;
    \item[(ii)]$t_2\geq 1$ and $\frac{l_1^p}{p}-\frac{\la d^{p^*}}{p^*}<0$;
    \item[(iii)]$t_2\geq 1$ and $\frac{l_1^p}{p}-\frac{\la d^{p^*}}{p^*}\geq 0$.
\end{enumerate}
\textbf{Case(i):} Define $g(t)= \phi_{u_0}(t) + \frac{l_1^pt^p}{p}+\frac{l_2^qt^q}{q}-\frac{\la d^{p^*}t^{p^*}}{p^*}$ for $t>0$. Then $\phi_{u_0}^\prime(1)<0$ and \eqref{weak-solI-13} says that $(l_1^p+l_2^q-d^{p^*})>0$. By \eqref{weak-solI-13}, we get $g^\prime(1)=0$ and $t_2<1$, $p<q<p^*$ along with \eqref{weak-solI-11} gives 
\[g^\prime(t_2) >t_2^{q-1}(l_1^p+l_2^q-\la d^{p^*})>0.\]
So $g$ is increasing in $[t_2,1]$. This helps us to obtain from \eqref{weak-solI-12} that
\begin{align*}
    m_\la^+ = g(1) \geq g(t_2) \geq \phi_{u_0}(t_2) +\frac{t_2^q}{q}(l_1^p+l_2^q-\la d^{p^*})>\phi_{u_0}(t_2)> \phi_{u_0}(t_1)\geq m_\la^+
\end{align*}
which is a contradiction.\\
\textbf{Case(ii):} We have $l_1^p\leq d^{p^*}$ from \eqref{weak-solI-13}. Using the definition of $S_2$, we have 
\begin{align*}
     &S_{2}^{\frac{p^*}{p}}\frac{l_1^pp^*}{p}< \la S_2^{\frac{p^*}{p}} d^{p^*}\leq \la  S_2^{\frac{p^*}{p}}|b|_\infty \lim_{k\to \infty}\int_\Om |w_k|^{p^*}~dx \\
     &\leq S_2^{\frac{p^*}{p}}|b|_\infty S_2^{-\frac{p^*}{p}} \left(\lim_{k\to \infty}\int_\Om H(\nabla w_k)^p~dx \right)^{\frac{p^*}{p}}= \la |b|_\infty l_1^{p^*} .
\end{align*}
This easily gives that 
\[l_1^p>\left(\frac{p^*S_{2}^{\frac{p^*}{p}}}{\la|b|_\infty p}\right)^{\frac{p}{p^*-p}}.\]
Now recalling definition of $\Lambda_{0}, \Lambda_{**}$ and $\la \in (0,\Lambda_{**})$, we find that 
\[\sup\left\{\int_\Om H(\nabla u)^p~dx:\; u\in M_\la^+\right\}\leq \left(\frac{p^*S_{2}^{\frac{p^*}{p}}}{\la|b|_\infty p}\right)^{\frac{p}{p^*-p}}<l_1^p \leq   \sup\left\{\int_\Om H(\nabla u)^p~dx:\; u\in M_\la^+\right\}\]
which is a contradiction again.\\
Therefore only Case(iii) is true and in this case, from \eqref{weak-solI-12} we have 
\begin{align*}
    m_\la^+ = I_\la(u_0)+\frac{l_1^p}{p}+\frac{l_2^q}{q}-\frac{\la d^{p^*}}{p^*} \geq \phi_{u_0}(1) \geq \phi_{u_0}(t_2) \geq m_\la^+.
\end{align*}
This straightaway states that 
\[t_2=1 \;\text{and}\; \frac{l_1^p}{p}+\frac{l_2^q}{q}-\frac{d^{p^*}}{p^*}=0.\]
Now since $\phi_{u_0}^\prime(t_2)=\phi_{u_0}^\prime(1) =0$, \eqref{weak-solI-13} says that $l_1^p+l_2^q=d^{p^*}$. Putting this in above, we get 
\[\left(\frac{1}{p}-\frac{1}{p^*}\right)l_1^p +\left(\frac{1}{q}-\frac{1}{p^*}\right)l_2^q=0 \;\text{that is}\; l_1^p=0=l_2^q \]
and thus $d^{p^*}=0$ also. Since $\phi_{u_k}^{\prime \prime}(1)>0$, we get $\phi_{u_0}^{\prime\prime}(1)\geq 0$ but $\la<\la_*$ says that $\phi_{u_0}^{\prime\prime}(1)>0$. Thus $\phi_{u_0}^\prime(1)=0$ tells us $u_0\in M_\la^+$. Also from the convergence results, we have 
\[m_\la^+=\lim_{k\to \infty}I_\la(u_k)= I_\la(u_0).\]
\end{proof}

\begin{Theorem}\label{weak-sol-crit}
$u_0$ defined in Proposition \ref{strong-con} is a weak solution of \eqref{maineqn} when $\la\in (0, \Lambda_{**})$ and $\alpha=p^*-1$.
\end{Theorem}
\begin{proof}
We first show that $u_0>0$ a.e. in $\Om$ via contradiction. We already know that $u_0\geq 0$ a.e. in $\Om$ and suppose that $u_0 \equiv 0$ on $K$ where $K\subset \Om$ has positive measure. Due to Lemma \ref{p3.6}, we can choose $0<h\in W^{1,\mathcal H}_0(\Om)$ such that $(u_0+th)^{1-\delta}>u_0^{1-\delta}$ in $\Om \setminus K$ for $t\in [0,\beta]$ and 
\begin{align*}
    0 & \leq I_\la(u_0+th)-I_\la(u_0)\\
    & =\frac{1}{p}\int_\Om (H(\nabla(u_0+th))^p-H(\nabla u_0))~dx + \frac{1}{q}\int_\Om (H(\nabla(u_0+th))^q-H(\nabla u_0)^q)~dx\\
    &\quad -\frac{1}{1-\delta}\int_\Om ((u_0+th)^{1-\delta}-u_0^{1-\delta})~dx -\frac{1}{p^*}\int_\Om ((u_0+th)^{p^*}-u_0^{p^*})~dx\\
    & < \frac{1}{p}\int_\Om (H(\nabla(u_0+th))^p-H(\nabla u_0))~dx + \frac{1}{q}\int_\Om (H(\nabla(u_0+th))^q-H(\nabla u_0)^q)~dx\\
    &\quad -\frac{t^{1-\delta}}{1-\delta}\int_K h^{1-\delta}~dx -\frac{1}{p^*}\int_\Om ((u_0+th)^{p^*}-u_0^{p^*})~dx.
\end{align*}
If we divide the above equation by $t>0$ and then pass the limit as $t\to 0^+$ then we obtain 
\[ 0 \leq \lim_{t\to 0^+}\frac{I_\la(u_0+th)-I_\la(u_0)}{t}=- \infty\]
which is a contradiction implying that $u_0>0$ a.e. in $\Omega$.

From Proposition \ref{strong-con}, we know that $u_k\to u_0$ strongly, $u_0\in M_\la^+$ with $m_\la^+$ achieved by $u_0$. Moreover from Lemma \ref{weak-solI}, for any $h\in W^{1,\mathcal H}_0(\Omega)$ we have 
\begin{align*}
   &\lim_{k\to \infty} \int_{\Om}\left(H(\nabla u_k)^{p-1}+\mu(x)H(\nabla u_k)^{q-1}\right)\nabla_{\eta}H(\nabla u_k)\nabla h\,dx\\
   &\quad -\int_{\Om}\left(a(x)u_k^{-\delta}-b(x)u_k^{p^*-1}\right)h\,dx=0,
\end{align*}
\begin{align*}
   \implies &\int_{\Om}\left(H(\nabla u_0)^{p-1}+\mu(x)H(\nabla u_0)^{q-1}\right)\nabla_{\eta}H(\nabla u_0)\nabla h\,dx\\
   &\quad -\int_{\Om}\left(a(x)u_0^{-\delta}-b(x)u_0^{p^*-1}\right)h\,dx=0
\end{align*}which clearly says that $u_0$ is a weak solution to \eqref{maineqn} with $\alpha=p^*-1$, refer \eqref{wksol}. This completes the proof.
\end{proof}


\textbf{Proof of Theorem \ref{cthm}:} The proof can be easily accomplished by combining Proposition \ref{strong-con} and Theorem \ref{weak-sol-crit}.

\end{document}